\documentclass[11pt]{amsart}
\usepackage{amssymb}
\usepackage{amsfonts} 
\usepackage{amsmath}                           
\usepackage{amsthm}
\usepackage{hyperref}
\usepackage[verbose]{geometry}

\usepackage[english]{babel}
\usepackage[latin1]{inputenc}
\usepackage{amsmath,amssymb,amsthm,epsfig}
\usepackage{eufrak}
\usepackage{color}
\usepackage{mathrsfs}
\usepackage{latexsym}
\usepackage{comment} 
\usepackage{bigints} 

\theoremstyle{plain}
   
\newtheorem{teo}{Theorem}[section]
\newtheorem{conj}{Conjecture}[section]
\newtheorem{prop}{Proposition}[section]
\newtheorem{defi}{Definition}[section]

\newtheorem{cor}{Corollary}[section]
\theoremstyle{definition}
\newtheorem{exe}{Example}[section]
\newtheorem{obs}{Remark}[section]
\begin{document}
\title[On Bott's residue formula for toric complete intersections]
{On Bott's residue formula for toric complete intersections}

\author{Miguel Rodr\'iguez Pe\~na}
\address{Arnulfo Miguel Rodr\'iguez Pe\~na \\
ICEX-UFMG, Departamento de Matem\'atica,
Av. Ant\^onio Carlos, 6627, 31270-901  Belo Horizonte, MG, Brazil.}
\email{amrp2024@ufmg.br}

\date\today

\begin{abstract}
We determine the number of singularities--counted with multiplicities--of generic distributions of dimension and codimension 
one on smooth complete intersections in compact toric orbifolds with isolated singularities. We also present some applications of these results. First, we analyze the case of regular distributions. As a second application, we establish a Poincar\'e-type inequality that relates the multidegree of a foliation to the multidegrees of an invariant smooth complete intersection curve. Our results are illustrated through several representative examples.
\end{abstract}

\maketitle


\section{Introduction}

The aim of this paper is to study holomorphic distributions of dimension 
and codimension one on smooth complete intersections in a compact toric orbifold through Bott's residue formula. 
For this we will consider homogeneous coordinates, where effective computations are feasible, see \cite{Cox}.
The study of holomorphic distributions and foliations on toric varieties is a recent topic and there are many authors studying it;
we refer to three recent works, namely \cite{FJM}, \cite{Mig2} and \cite{WWA}.
We remark that distributions and foliations have been studied by several authors on complex projective varieties, 
see for instance \cite{Ca,OmCoJa,CCM,GJM,Iza,Jou}. \\

By invoking Bott's residue formula (in the orbifold context), see \cite{BB,Mig,Ding,Sata,Iza,Jou,Mig2}, we deduce the number of singularities, counting multiplicities, of generic distributions of dimension and codimension one on smooth complete intersections in a compact toric orbifold with isolated singularities; see for instance Theorems \ref{aa}, \ref{teo1}, \ref{tte}.
Later, we give some applications of our formulas. First we study the regular distributions; see for instance Theorem \ref{ttee},
Corollaries \ref{cwp}, \ref{coror}, Theorem \ref{teo2}, and Examples \ref{exee}, \ref{eexe}. As a second application, we give a 
Poincar\'e-type inequality, relating the multidegree of a foliation with the multidegrees of an invariant smooth complete intersection;
see Corollary \ref{cccor}, Theorem \ref{teo3}, and Corollary \ref{ccrr}.
We exemplify our results with a range of illustrative examples. \\

In the real context, if $M$ be a closed connected smooth manifold, Thurston showed in \cite{Tu} that 
$M$ admits a $C^{\infty}$ regular codimension one foliation if and only if $\chi(M)=0$.
Regular distributions are quite rare, for example, on the complex projective space, a holomorphic distribution $\mathcal{D}$ 
on $\mathbb{P}^{n}$ is non singular if and only if $n$ is odd, $\mathrm{codim}(\mathcal{D})=1$ and $\deg(\mathcal{D})=0$; see
\cite[p.36]{Coo} and \cite{GHS}. For surfaces, M. Brunella completely classified regular foliations on complex projective surfaces $S$ with Kodaira dimension $\kappa(S)<2$. For surfaces of general type, the classification is open. As a corollary of this classification we have:

\begin{cor} \cite{Bru}
Let $\mathcal{F}$ be a regular one dimensional holomorphic foliation on a smooth projective rational surface $S$. 
Then $S$ is a Hirzebruch surface and $\mathcal{F}$ is induced by a $\mathbb{P}^{1}$-fibration $S\rightarrow\mathbb{P}^{1}$.
\end{cor}

In higher dimensions, there is no classification of regular holomorphic foliations. 
However, Touzet proposed the following generalization the corollary above:

\begin{conj} \cite{Dru}
Let $\mathcal{F}$  be a regular holomorphic foliation on a rationally connected projective manifold $X$. 
Then $\mathcal{F}$ is induced by a fibration $X\rightarrow Y$ onto a projective manifold.
\end{conj}

The conjecture was verified for weak Fano manifolds in \cite{Dru}, and is 
open already in dimension $3$. However, for a regular codimension one foliation on
a projective threefold, the Minimal Model Program can be used to greatly reduce the problem; the conjecture is proved 
in some special cases, and reduces the problem to understanding regular foliations with nef canonical class. 
For more details see \cite{AF,Dru,Fig}. It is important to note that the class of regular and nonintegrable codimension one distributions includes contact structures. Contact structures are currently attracting much interest, both from mathematicians and physicists, 
and their classification is still in progress. In \cite{Dru2} Druel proved that toric contact manifolds are either $\mathbb{P}^{2n+1}$ 
or $\mathbb{P}(\mathcal{T}_{\mathbb{P}^1\times\cdots\times\mathbb{P}^1})$. 
Several authors have studied complex contact structures, we refer \cite{Ca,Dru2,KPS} and references therein. \\

Henri Poincar\'e studied in \cite{Poin} the problem to decide whether a holomorphic
foliation $\mathcal{F}$ on the complex projective plane $\mathbb{P}^2$ admits a rational first integral. Poincar\'e
observed that, in order to solve this problem, it is sufficient to find a bound for the degree of the generic $\mathcal{F}$-invariant curves.
Determining such a bound is known as the \textit{Poincar\'e problem}. Although it is well-known that such a bound does not exist in
general, under certain hypotheses, there are several works about Poincar\'e problem and its generalizations; see for instance 
\cite{BruMe,Bru1,Can,Alc,EstKle,Gal,Pere,So}. It is important to point out that Corollary $\ref{ccrr}$ generalizes the result obtained in 
\cite{Mar} about the Poincar\'e problem for smooth complete intersection curves on projective spaces. 
For instance, this problem has also been studied in \cite{CCG,CJ,CE,Es,Mar}.

\section{Preliminaries}

We first fix some notation. For more details about toric varieties see \cite{Cox2, Cox, Fu1}.
Let $N\simeq\mathbb{Z}^n$ be a free $\mathbb{Z}$-module of rank $n$ and 
$M=\mathrm{Hom}_{\mathbb{Z}}(N, \mathbb{Z})$ be its dual. Let $\Delta$ be a rational complete simplicial $n$-dimensional fan in 
$N_{\mathbb{R}}$, and $\mathbb{P}_{\Delta}$ the $n$-dimensional compact toric orbifold associated, i.e., a complex variety with at most quotient singularities; see Section $\ref{orbif}$.
 
The one-dimensional cones of $\Delta$ form the set $\Delta(1)$. We will assume that $\Delta(1)$ spans $N_{\mathbb{R}}$.
Let $S(\Delta)=\mathbb{C}\left[z_1,\ldots,z_{n+r}\right]$ be the polynomial ring over $\mathbb{C}$ with variables $z_1,\ldots,z_{n+r}$
corresponding to the integral generators $\rho_1,\ldots,\rho_{n+r}\in\Delta(1)$.
For each cone $\sigma \in \Delta$, we get the monomial $z_{\hat{\sigma}}=\prod_{\rho_i\notin \sigma(1)}z_i$, and let
$\mathcal{Z}=V(\{z_{\hat{\sigma}},\,\sigma \in \Delta\})\subset \mathbb{C}^{n+r}$ be the exceptional variety.
The Weil divisor class group $\mathcal{A}_{n-1}(\mathbb{P}_{\Delta})$ is a finitely generated 
abelian group of rank $r$.
Moreover, the group $G=Hom_{\mathbb{Z}}(\mathcal{A}_{n-1}(\mathbb{P}_{\Delta}), \mathbb{C}^*)$ 
acts naturally on $\mathbb{C}^{n+r}$ and leaves $\mathcal{Z}$ invariant.
A toric orbifold $\mathbb{P}_{\Delta}$ can be described as a geometric quotient, see \cite{Cox}:
$$\mathbb{P}_{\Delta}=\frac{\mathbb{C}^{n+r}-\mathcal{Z}}{G}.$$

Each variable $z_i$ in the coordinate ring $S(\Delta)$ corresponds to a torus invariant
irreducible divisor $D_i$ of $\mathbb{P}_{\Delta}$.
The degree of a monomial $\prod_{i=1}^n z_i^{a_i}$ is defined by
$\left[\sum_{i=1}^na_iD_i\right]\in \mathcal{A}_{n-1}(\mathbb{P}_{\Delta})$. 
The polynomial graded ring $S(\Delta)$ is called homogeneous coordinate ring of the toric variety $\mathbb{P}_{\Delta}$.
A polynomial $f$ in the graded piece $S_{\alpha}$ corresponding to $\alpha\in\mathcal{A}_{n-1}(\mathbb{P}_{\Delta})$ is said
to be quasi-homogeneous of degree $\alpha$.

Denote by $\mathcal{O}_{\mathbb{P}_{\Delta}}$ the structure sheaf of $\mathbb{P}_{\Delta}$. Let $\mathcal{O}_{\mathbb{P}_{\Delta}}(D)$ be the coherent sheaf on $\mathbb{P}_{\Delta}$ determined by a Weil divisor $D$. If $\alpha = [D] \in \mathcal{A}_{n-1}(\mathbb{P}_{\Delta})$, then
$$S_{\alpha}\simeq \mathrm{H}^0(\mathbb{P}_{\Delta},\mathcal{O}_{\mathbb{P}_{\Delta}}(D)).$$
If $\mathbb{P}_{\Delta}$ is a complete toric variety, then $S_{\alpha}$ is finite dimensional for every $\alpha$, 
and in particular, $S_0 = \mathbb{C}$.

Let $f_1,\ldots,f_m$ be quasi-homogeneous polynomials, and set $a_i=\deg(f_i)$. They define a closed subset 
$V=V(a_1,\ldots,a_m)$ of $\mathbb{P}_{\Delta}$, because $\mathbb{P}_{\Delta}$ is a geometric quotient.
We say that $V$ is a quasi-smooth complete intersection if $V\cap(\mathbb{C}^{n+r}-\mathcal{Z})$ 
is either empty or a smooth subvariety of codimension $m$ in $(\mathbb{C}^{n+r}-\mathcal{Z})$.
We can relate this notion to a suborbifold; $V$ is a quasi-smooth complete intersection
if and only if $V$ is a suborbifold of $\mathbb{P}_{\Delta}$, for details see \cite{Cox1, Ma}.

Given $\rho_j\in\Delta(1)$, denote by $n_j$ the unique generator of $\rho_j\cap N$. 
Consider the $r$ linearly independent over $\mathbb{Z}$ relations between the $n_{\rho_j}$ 
of type $\sum_ja_{ij} n_j=0$, $i=1,\ldots,r$. 
Then, there are $r$ vector fields $R_i=\sum_ja_{ij}z_j\frac{\partial}{\partial z_j}$, $i=1,\ldots,r$, 
tangent to the orbits of $G$ and $\mathrm{Lie}(G)=\langle R_1,\dots,R_r\rangle$. 
We will call these vector fields $R_i$, the radial vector fields on $\mathbb{P}_{\Delta}$, see \cite{Cox1}. \\

Let $\mathbb{P}_{\Delta}$ be a complete simplicial toric variety  of dimension $n$ where $\Delta(1)$
spans $N_{\mathbb{R}}$. Given $h_i:=\left[D_i\right]\in\mathcal{A}_{n-1}(\mathbb{P}_{\Delta})$, $i=1,\ldots,n+r$, we will denote by 
$\texttt{C}_{j}(h)$ the $j$th elementary symmetric function of the variables $h_1,\ldots, h_{n+r}$. 

\subsection{Examples}

\begin{enumerate}
	\item \label{exe1} Weighted projective spaces. \cite{Cox2} Let $\omega_0,\dots,\omega_n$ be positive integers with 
$gcd(\omega_0,\dots,\omega_n)=1$. Set $\omega=(\omega_{0},\dots,\omega_{n})$. 
Consider the $\mathbb{C}^*$ action on $\mathbb{C}^{n+1}$ given by 
$$t\cdot(z_0,\dots,z_n)=(t^{\omega_0} z_0,\dots, t^{\omega_n} z_n).$$
Then, the weighted projective spaces is defined by
$$\mathbb{P}(\omega) = \frac{\mathbb{C}^{n+1} - \{0\}}{\mathbb{C}^*}.$$ 
If $\omega_0=\cdots=\omega_n=1$, then $\mathbb{P}(\omega)=\mathbb{P}^{n}$. When $\omega_0,\dots,\omega_n$ 
are pairwise coprime, we say that $\mathbb{P}(\omega)$ is a well formed weighted projective space and we have 
$$\mathrm{Sing}(\mathbb{P}(\omega))=\left\{\bar{e}_i\,|\,\omega_i>1\right\}.$$

Moreover, we have $\mathcal{A}_{n-1}(\mathbb{P}(\omega))\simeq\mathbb{Z}$ and $\deg(z_i)=\omega_i$ for all $0\leq i\leq n+1$.
Consequently the homogeneous coordinate ring of $\mathbb{P}(\omega)$ is given by
$\mathrm{S}=\oplus_{\alpha\geq 0}\mathrm{S}_{\alpha}$, where 
$$\mathrm{S}_{\alpha}=\bigoplus_{p_0\omega_0+\cdots+p_n\omega_n=\alpha} \mathbb{C}\cdot {z_0}^{p_0}\dots\,{z_n}^{p_n}.$$

\vskip 0.2cm

 \item \label{exe2} Multiprojective spaces. \cite{Cox2}
Consider $\mathbb{C}^{n+1}\times\mathbb{C}^{m+1}$ in coordinates $(z_1,z_2)=(z_{1,0},\ldots,z_{1,n},$ $z_{2,0},\ldots,z_{2,m})$. 
Now, consider the $\mathbb{C}^*\times\mathbb{C}^*$ action on $\mathbb{C}^{n+1}\times\mathbb{C}^{m+1}$ given by 
$$(\mu,\lambda)\cdot(z_1,z_2)=(\mu z_{1,0},\dots,\mu z_{1,n},\lambda z_{2,0},\dots,\lambda z_{2,m}).$$
Then, the multiprojective space is defined by 
$$\mathbb{P}^n\times\mathbb{P}^m=\frac{\mathbb{C}^{n+1}\times\mathbb{C}^{m+1}-\mathcal{Z}}{\mathbb{C}^*\times\mathbb{C}^*},$$
where $\mathcal{Z}=\left(\left\{0\right\}\times\mathbb{C}^{m+1}\right) \cup \left(\mathbb{C}^{n+1}\times\left\{0\right\}\right)$.
Moreover, we have $\mathcal{A}_{n+m-1}(\mathbb{P}^n\times\mathbb{P}^m)\simeq\mathbb{Z}^2$,
$\deg(z_{1,i})=(1,0)$ for all $0\leq i\leq n$, and $\deg(z_{2,j})=(0,1)$ for all $0\leq j\leq m$.
Consequently the homogeneous coordinate ring of $\mathbb{P}^n\times\mathbb{P}^m$ is given by 
$\mathrm{S}=\oplus_{\alpha,\,\beta\geq 0}\mathrm{S}_{(\alpha,\beta)}$, where 
$$\mathrm{S}_{(\alpha,\beta)}=\bigoplus_{p_0+\cdots+p_n=\alpha\,;\,q_0+\cdots+q_m=\beta} \mathbb{C}\cdot {z_{1,0}}^{p_0}\dots\,{z_{1,n}}^{p_n} 
{z_{2,0}}^{q_0}\dots\,{z_{2,m}}^{q_m},$$
is the ring of bihomogeneous polynomials of bidegree $(\alpha,\beta)$. More generally, we can define 
$\mathbb{P}^{n_1}\times\cdots\times\mathbb{P}^{n_k}$ (in general, a finite product of toric varieties is again a toric variety). \\

 \item \label{exe4} Rational normal scrolls.
Let $a_1,\ldots,a_n$ be integers. Consider the $\mathbb{C}^*\times\mathbb{C}^*$ action on $\mathbb{C}^2\times\mathbb{C}^n$ given by 
$$(\lambda,\mu)(z_{1,1},z_{1,2},z_{2,1},\dots,z_{2,n})=(\lambda z_{1,1}, \lambda z_{1,2}, \mu\lambda^{-a_1} z_{2,1},
\dots,\mu\lambda^{-a_n} z_{2,n}).$$
Then, an $n$-dimensional rational normal scroll $\mathbb{F}(a_1,\dots,a_n)$ is a projective toric manifold defined by 
$$\mathbb{F}(a_1,\dots,a_n)=\frac{\mathbb{C}^2\times\mathbb{C}^n-\mathcal{Z}}{\mathbb{C}^*\times\mathbb{C}^*},$$
where $\mathcal{Z}=\left(\left\{0\right\}\times\mathbb{C}^{n}\right) \cup \left(\mathbb{C}^{2}\times\left\{0\right\}\right)$.
Moreover, we have $\mathcal{A}_{n-1}(\mathbb{F}(a_1,\dots,a_n))\simeq\mathbb{Z}^2$ and the homogeneous coordinate ring of 
$\mathbb{F}(a_1,\dots,a_n)$ is given by $\mathrm{S}=\oplus_{\alpha\in\mathbb{Z},\,\beta\geq 0}\mathrm{S}_{(\alpha,\beta)}$, where 
$$\mathrm{S}_{(\alpha,\beta)}=\bigoplus_{\alpha=p_1+p_2 - \sum_iq_ia_i\,;
\,\beta=\sum_iq_i} \mathbb{C}\cdot {z_{1,1}}^{p_1}{z_{1,2}}^{p_2}{z_{2,1}}^{q_1}\dots\,{z_{2,n}}^{q_n}.$$
In particular, $\deg(z_{1,1})=\deg(z_{1,2})=(1,0)$ and $\deg(z_{2,i})=(-a_i,1)$.

Consider  $1\leq a_1 \leq a_2 \leq \dots \leq a_n$ integers, it is possible to show that $\mathbb{F}(a_1,\dots,a_n)\simeq\mathbb{P}(\mathcal{O}_{\mathbb{P}^1}(a_1)\oplus\dots\oplus\mathcal{O}_{\mathbb{P}^1}(a_n))$. For more details see \cite{Cox2, Re}.
\end{enumerate}

\vspace{0,1cm}

\subsection{Orbifolds} \label{orbif}\cite{Rua, Bla, Sata, Man}. 
An $n$-dimensional orbifold $X$ is a complex space endowed with the following property: each point $x\in X$ possesses an open neighborhood, which is the quotient $U_x=\widetilde{U}/G_x$, where $\widetilde{U}$
is a complex manifold of dimension $n$, and $G_x$ is a properly discontinuous finite group of automorphisms of $\widetilde{U}$.
Thus, locally we have a quotient map $\pi_x:(\widetilde{U},\tilde{x})\rightarrow(U_x,x)$.

Since the possible singularities appearing in an orbifold $X$ are quotient singularities, we have that $X$ is reduced, normal, Cohen-Macaulay,  and with only rational singularities. 

Satake's fundamental idea was to use local smoothing coverings to extend the
definition of usual known objects to orbifolds. For instance, smooth differential $k$-forms
on $X$ are $C^{\infty}$-differential $k$-forms $\omega$ on $X_{\mathrm{reg}}=X \setminus \mathrm{Sing}(X)$ such that the pull-back
$\pi^{\ast}_x\omega$ extends to a $C^{\infty}$-differential $k$-form on every local smoothing covering 
$\pi_x:(\widetilde{U},\tilde{x})\rightarrow(X,x)$ of $X$. Hence, if $\omega$ is a smooth $2n$-form on $X$ with compact support 
$\mathrm{Supp}(\omega)\subset(X, x)$, then by definition
$$\int_X^{orb} \omega=\frac{1}{\# G_x}\int_{\widetilde{U}}\pi^{\ast}\omega.$$

If now $\omega$ has compact support, we use a partition of unity 
$\left\{\rho_{\alpha},U_{\alpha}\right\}_{\alpha\in\Lambda}$, where
$\pi_{\alpha}:(\widetilde{U}_{\alpha},\tilde{x}_{\alpha})\rightarrow(U_{\alpha},x_{\alpha})$ is a local smoothing covering and
$\sum \rho_{\alpha}(x)=1$ for all $x\in\mathrm{Supp}(\omega)$, and set
$$\int_X^{orb} \omega=\sum_{\alpha}\int_{X}^{orb}\rho_{\alpha}\omega.$$

\begin{obs}\cite{Man} Let $X$ be a compact orbifold and 
$Ker(X)=\left\{g\in\coprod_{\alpha\in\Lambda}G_{\alpha}:\,g\cdot x=x,\,\forall x\in X\right\}$. Then 
$$\int_X^{orb} \omega=\frac{1}{\# Ker(X)}\int_{X_{reg}}\omega.$$
\end{obs}

Analogous to the manifold case, we have the sheaf of $C^{\infty}$-differential $k$-forms and
exterior differentiation, moreover, the concepts of connection and curvature are well defined. 
Stokes' formula, Poincar\'e's duality and Rham's theorem also hold on orbifolds.

\subsection{One-dimensional foliation and codimension one distribution} \cite{Mau, Jou, Mig2, Ro}.
There exists an exact sequence known as the generalized Euler's sequence 
$$0\rightarrow \mathcal{O}_{\mathbb{P}_{\Delta}}^{\oplus
r}\rightarrow\bigoplus_{i=1}^{n+r}\mathcal{O}_{\mathbb{P}_{\Delta}}(D_i)\rightarrow
\mathcal{T}\mathbb{P}_{\Delta}\rightarrow0,$$
\noindent where
$\mathcal{T}\mathbb{P}_{\Delta}=\mathcal{H}om(\Omega_{\mathbb{P}_{\Delta}}^1,\mathcal{O}_{\mathbb{P}_{\Delta}})$ is the so-called Zariski tangent sheaf of $\mathbb{P}_{\Delta}$. 
Let $i:\mathbb{P}_{\Delta \,\mathrm{reg}} \rightarrow \mathbb{P}_{\Delta}$
be the inclusion of the regular part $\mathbb{P}_{\Delta \,\mathrm{reg}}= \mathbb{P}_{\Delta} - \mathrm{Sing}(\mathbb{P}_{\Delta})$.
Since $\mathbb{P}_{\Delta}$ is a complex orbifold then $\mathcal{T}\mathbb{P}_{\Delta} \simeq 
i_{\ast}\mathcal{T}\mathbb{P}_{\Delta \,\mathrm{reg}}$, where $\mathcal{T}\mathbb{P}_{\Delta \,\mathrm{reg}}$ is the tangent sheaf of $\mathbb{P}_{\Delta \,\mathrm{reg}}$; 
see \cite[Appendix A.2]{Cox3}.\\

Let $\mathcal{O}_{\mathbb{P}_{\Delta}}(d)=\mathcal{O}_{\mathbb{P}_{\Delta}}(d_1,\ldots,d_{n+r})
=\mathcal{O}_{\mathbb{P}_{\Delta}}(\sum_{i=1}^{n+r}d_{i}D_{i})$, 
where $\sum_{i=1}^{n+r}d_{i}D_{i}$ is a Cartier divisor. \\ 
Set
$\Omega^1_{\mathbb{P}_{\Delta}}(d)=\Omega^1_{\mathbb{P}_{\Delta}}\otimes \mathcal{O}_{\mathbb{P}_{\Delta}}(d)$.
Tensorizing the dual Euler's sequence by $\mathcal{O}_{\mathbb{P}_{\Delta}}(d)$ we get
\begin{eqnarray*}
0 \rightarrow \Omega^1_{\mathbb{P}_{\Delta}}(d)
\rightarrow \bigoplus_{i=1}^{n+r}\mathcal{O}_{\mathbb{P}_{\Delta}}(d_1,\dots,d_i-1,\dots,d_{n+r})
\rightarrow Cl(\mathbb{P}_{\Delta})\otimes_{\mathbb{Z}}\mathcal{O}_{\mathbb{P}_{\Delta}}(d) \rightarrow 0.
\end{eqnarray*}
For more details, see for instance \cite{Cox1, Cox2}.

\begin{defi} 
A one-dimensional holomorphic foliation $\mathcal{F}$ on $\mathbb{P}_{\Delta}$ of degree 
$\left[\sum_{i=1}^{n+r}d_{i}D_{i}\right]\in\mathcal{A}_{n-1}(\mathbb{P}_{\Delta})$
is a global section of $\mathcal{T}\mathbb{P}_{\Delta}(d_1,\dots,d_{n+r})$.
For simplicity of notation we say that $\mathcal{F}$ has degree $(d_1,\dots,d_{n+r})$.
We will consider one-dimensional holomorphic foliations whose singular scheme has codimension greater than $1$.
\end{defi}

If $H^1\left(\mathbb{P}_{\Delta}, \mathcal{O}_{\mathbb{P}_{\Delta}}
(d_1,\dots,d_{n+r})\right)=0$; see for instance the Demazure vanishing theorem \cite[Theorem $9.2.3$]{Cox2},
then, we have that a one-dimensional holomorphic foliation $\mathcal{F}$ on
$\mathbb{P}_{\Delta}$ of degree $(d_1,\dots,d_{n+r})$ is given by a
polynomial vector field in homogeneous coordinates of the form
$$
X=\sum_{i=1}^{n+r}P_i\frac{\partial}{\partial z_i},
$$
where $P_i$ is a polynomial of degree $(d_1,\dots,d_i+1,\dots,d_{n+r})$
for all $i=1,\dots,n+r$, modulo addition of a vector field of the form
$\sum_{i=1}^{r}g_iR_i$, where $R_1,\dots,R_r$ are the radial vector fields on $\mathbb{P}_{\Delta}$.
We say that $X$ is a quasi-homogeneous vector field. Moreover we have  
$$\mathrm{Sing}(\mathcal{F})=\pi\left(\left\{p\in\mathbb{C}^{n+r} : (R_1\wedge\cdots\wedge R_r\wedge X)(p)=0\right\}\right),$$
where $\pi:\mathbb{C}^{n+r}-{\mathcal{Z}} \rightarrow \mathbb{P}_{\Delta}$ is the canonical projection. 

Let $\mathcal{F}$ be a foliation on $\mathbb{P}_{\Delta}$ and $V = \{f = 0\}$ a quasi-homogeneous
hypersurface. We recall that $V$ is invariant by $\mathcal{F}$ if and only if 
$X(f) = g\cdot f$, where $X$ is a quasi-homogeneous vector field which defines $\mathcal{F}$ in homogeneous coordinates.

We will say that a holomorphic foliation $\mathcal{F}$ is generic if it has at most isolated singularities, 
and that $\mathcal{F}$ is regular if $\mathrm{Sing}(\mathcal{F})=\varnothing$. 

\begin{defi} 
A codimension one holomorphic distribution $\mathcal{D}$ on $\mathbb{P}_{\Delta}$ of degree 
$d=\sum_{i=1}^{n+r}d_{i}\left[D_{i}\right]\in\mathcal{A}_{n-1}(\mathbb{P}_{\Delta})$
is a global section $\omega$ of $\Omega^1_{\mathbb{P}_{\Delta}}(d)$.  
For simplicity of notation we say that $\mathcal{D}$ has degree $d=(d_1,\dots,d_{n+r})$.
A codimension one holomorphic foliation $\mathcal{F}$ on $\mathbb{P}_{\Delta}$ is a distribution that satisfies 
the Frobenius integrability condition $\omega\wedge d\omega=0$.
We will consider codimension one holomorphic foliations whose singular scheme has codimension greater than $1$.
\end{defi}

If $H^1\big(\mathbb{P}_{\Delta}, \Omega^1_{\mathbb{P}_{\Delta}}(d)\big)=0$; see for instance the Bott-Steenbrink-Danilov vanishing theorem
\cite[Theorem $9.3.1$]{Cox2}, then, we have that a codimension one holomorphic distribution $\mathcal{D}$ on
$\mathbb{P}_{\Delta}$ of degree $d$ is given by a polynomial form in homogeneous coordinates of the form
$$
\omega=\sum_{i=1}^{n+r}P_i \, dz_i,
$$
where $P_i$ is a quasi-homogeneous polynomial of degree $(d_1,\dots,d_i-1,\dots,d_{n+r})$ for all $i=1,\dots,n+r$, such that
$i_R \, \omega=0$ for all $R\in\mathrm{Lie}(G)$. 
Note that $\omega$ is a quasi-homogeneous form of degree $d$. We define the singular set of $\mathcal{D}$ by 
$$\mathrm{Sing}(\mathcal{D})=\pi\left(\left\{p\in\mathbb{C}^{n+r}:\,\omega(p)=0\right\}\right),$$
where $\pi:\mathbb{C}^{n+r}-{\mathcal{Z}} \rightarrow \mathbb{P}_{\Delta}$ is the canonical projection. 

Let $V = \{f = 0\}$ be a quasi-homogeneous hypersurface. We recall that $V$ is invariant by $\mathcal{D}$ if and only if 
there is a holomorphic $2$-form $\Theta$ in homogeneous coordinates, such that $\omega\wedge d f = f\cdot\Theta$. 

We will say that a codimension one distribution $\mathcal{D}$ 
is generic if it has at most isolated singularities, 
and that $\mathcal{D}$ is regular if $\mathrm{Sing}(\mathcal{D})=\varnothing$. 

\subsection{Examples}

\begin{enumerate}
	\item Weighted projective spaces.
The Euler's sequence on $\mathbb{P}(\omega)$ is an exact sequence of orbibundles
$$
0\longrightarrow
\underline{\mathbb{C}}\longrightarrow\bigoplus_{i=0}^{n
}\mathcal{O}_{\mathbb{P}(\omega)}(\omega_i) \longrightarrow
\mathcal{T}\mathbb{P}(\omega) \longrightarrow 0,
$$
where $\underline{\mathbb{C}}$ is the trivial line orbibundle on
$\mathbb{P}(\omega)$. The radial vector field is given by
$$R = \omega_0 z_0 \frac{\partial}{\partial z_0}+\cdots+ \omega_n z_n \frac{\partial}{\partial z_n}.$$
 \item Multiprojective spaces. The Euler's sequence on $\mathbb{P}^{n}\times\mathbb{P}^{m}$ is
$$0\longrightarrow \underline{\mathbb{C}^{2}}\,
{\longrightarrow}
\mathcal{O}_{\mathbb{P}^{n}\times\mathbb{P}^{m}}(1,0)^{\oplus n+1}\oplus\mathcal{O}_{\mathbb{P}^{n}\times\mathbb{P}^{m}}(0,1)^{\oplus m+1}
{\longrightarrow}\,
\mathcal{T}(\mathbb{P}^{n}\times\mathbb{P}^{m})
\longrightarrow 0.$$
Here, the radial vector fields are given by
$$R_1 = z_{1,0} \frac{\partial}{\partial z_{1,0}}+\cdots+ z_{1,n} \frac{\partial}{\partial z_{1,n}},\,\,\mbox{and}\,\,\,
R_2 = z_{2,0} \frac{\partial}{\partial z_{2,0}}+\cdots+ z_{2,m} \frac{\partial}{\partial z_{2,m}}.$$

 \item Rational normal scrolls. The Euler's sequence on $\mathbb{F}(a)=\mathbb{F}(a_1,\dots,a_n)$ is
$$0 \rightarrow\mathcal{O}_{\mathbb{F}(a)}^{\oplus
2}\rightarrow\mathcal{O}_{\mathbb{F}(a)}(1,0)^{\oplus 2}\oplus\bigoplus_{i=1}^n
\mathcal{O}_{\mathbb{F}(a)}(-a_i,1)\rightarrow \mathcal{T}\mathbb{F}(a)\rightarrow0.$$
Here, the radial vector fields are given by
$$R_1 = z_{1,1}\frac{\partial}{\partial z_{1,1}}+ z_{1,2}\frac{\partial}{\partial z_{1,2}} 
- \sum_{i=1}^{n}a_i z_{2,i}\frac{\partial}{\partial z_{2,i}},\,\,\mbox{and}\,\,\, 
R_2 = \sum_{i=1}^{n}z_{2,i}\frac{\partial}{\partial z_{2,i}}.$$
 \end{enumerate}

\bigskip

\begin{defi} \cite{Mig, Mig2} Let $\mathbb{P}_{\Delta}$ be an $n$-dimensional compact toric orbifold with at most isolated singularities.
\begin{enumerate}
	\item Let $\mathcal{F}$ be a one-dimensional foliation with isolated zeros and let $X$ be a holomorphic section of 
$\mathcal{T}\mathbb{P}_{\Delta}(d)$ that induces $\mathcal{F}$. The Poincar\'e-Hopf index of $\mathcal{F}$ at a zero $p$ is defined by
$$\mathcal{I}_p^{orb}\left(X\right)=\frac{1}{\#G_p}\mathrm{Res}_{\,\tilde{p}}\left\{\frac{\mathrm{Det}(J\tilde{X})}
{\tilde{X}_1\dots\tilde{X}_n}d\tilde{z}_1\wedge\dots\wedge d\tilde{z}_n\right\},$$
where $\pi_p:(\widetilde{U},\tilde{p})\rightarrow(U,p)$ denotes the local quotient map at $p$:
$\tilde{X}= \pi_p^\ast X$,
$J{\tilde{X}}=\big(\frac{\partial\tilde{X}_i}{\partial\tilde{z}_j}\big)_{1\leq i,\,j\leq n}$, and 
$\,\mathrm{Res}_{\,\tilde{p}}\displaystyle\left\{\frac{\mathrm{Det}\big(J\tilde{X}\,\big)}
{\tilde{X}_1\dots\tilde{X}_n}d\tilde{z}_1\wedge\dots\wedge d\tilde{z}_n\right\}$ is Grothendieck's point residue.\\

  \item Similarly, let $\mathcal{D}$ be a codimension one distribution with isolated zeros and let $\omega$ be a holomorphic section of 
$\Omega^1_{\mathbb{P}_{\Delta}}(d)$ that induces $\mathcal{D}$. The Poincar\'e-Hopf index of $\mathcal{D}$ 
at a zero $p$ is defined by
$$\mathcal{I}_p^{orb}\left(\mathcal{D}\right)=\frac{1}{\#G_p}\mathrm{Res}_{\,\tilde{p}}\left\{\frac{\mathrm{Det}(J\tilde{\omega})}
{\tilde{\omega}_1\dots\tilde{\omega}_n}d\tilde{z}_1\wedge\dots\wedge d\tilde{z}_n\right\},$$
where $\pi_p:(\widetilde{U},\tilde{p})\rightarrow(U,p)$ denotes the local quotient map at $p$:
$\tilde{\omega}= \pi_p^\ast \omega$,
$J{\tilde{\omega}}=\big(\frac{\partial\tilde{\omega}_i}{\partial\tilde{z}_j}\big)_{1\leq i,\,j\leq n}$, and 
$\,\mathrm{Res}_{\,\tilde{p}}\displaystyle\left\{\frac{\mathrm{Det}\big(J\tilde{\omega}\,\big)}
{\tilde{\omega}_1\dots\tilde{\omega}_n}d\tilde{z}_1\wedge\dots\wedge d\tilde{z}_n\right\}$ is Grothendieck's point residue.
\end{enumerate}
\end{defi}

This index was introduced by Satake \cite{Sata} for $C^{\infty}$-vector fields in the orbifold case, see also \cite{Ding, Iza, Jou}.
The Grothendieck point residue embodies the Poincar\'e-Hopf index, the Milnor number and the intersection number of $n$ divisors in
$\mathbb{C}^n$ which intersect properly, and has many uses in deep results such as the Baum-Bott theorem, which is a
generalization of both the Poincar\'e-Hopf theorem and the Gauss-Bonnet theorem in the complex realm, see \cite{Grif, SoR, Su}. 


\section{Bott's residue formula for one-dimensional foliations in compact toric orbifolds.}

Set $k=(k_1,\ldots,k_n)\in\mathbb{N}^{n}$, $\left|k\right|=\sum k_i$ and $\binom{m}{k}=\frac{m!}{k_1!\cdots k_{n}!}$.

\begin{prop} \label{pro1} Let $\mathbb{P}_{\Delta}$ be an $n$-dimensional compact toric orbifold with isolated singularities.
Let $\mathcal{F}$ be a one-dimensional foliation of degree $d=\sum_{i}d_{i} h_i$
with isolated singularities on $\mathbb{P}_{\Delta}$. Then, the singular scheme of $\mathcal{F}$ consists of
$$\#\,\mathrm{Sing}(\mathcal{F})
=\sum_{j=0}^{n}\left\{\sum_{\left|k\right|=n-j}\binom{n-j}{k} \int_{\mathbb{P}_{\Delta}}^{\,orb} 
\texttt{C}_j(h)\cdot \prod_i \big(d_ih_i\big)^{k_i}\right\},$$ 
points counted with multiplicity, where $\texttt{C}_j(h)$ is the $j$th elementary symmetric function
of the variables $h_1,\ldots,h_{n+r}$.
\end{prop}


\begin{proof} It follows from theorem \cite[Theorem 1]{Mig} that
$$\#\,\mathrm{Sing}(\mathcal{F})
=\int_{\mathbb{P}_{\Delta}}^{\,orb} c_{n}\big(\mathcal{T}\mathbb{P}_{\Delta}\otimes\mathcal{O}_{\mathbb{P}_{\Delta}}(d)\big).$$

\noindent Using the generalized Euler's sequence
$$0\rightarrow \mathcal{O}_{\mathbb{P}_{\Delta}}^{\oplus r}\rightarrow\bigoplus_{i=1}^{n+r}\mathcal{O}_{\mathbb{P}_{\Delta}}(D_i)
\rightarrow\mathcal{T}\mathbb{P}_{\Delta}\rightarrow0,$$
we have 
$$c\left(\mathcal{T}\mathbb{P}_{\Delta}\right)
=c\left(\mathcal{O}^{\oplus r}_{\mathbb{P}_{\Delta}}\right)c\left(\mathcal{T}\mathbb{P}_{\Delta}\right)
=c\left(\bigoplus_{i=1}^{n+r} \mathcal{O}_{\mathbb{P}_{\Delta}}(D_i)\right),$$
where $c$ denotes the total Chern class. 
Note that $h_i=c_{1}(\mathcal{O}_{\mathbb{P}_{\Delta}}(D_i))=\left[D_i\right]$. Then
\begin{eqnarray*}
c_{n}\big(\mathcal{T}\mathbb{P}_{\Delta}\otimes\mathcal{O}_{\mathbb{P}_{\Delta}}(d)\big) &=&
\sum_{j=0}^{n}c_{j}\big(\mathcal{T}\mathbb{P}_{\Delta}\big)c_{1}\big(\mathcal{O}_{\mathbb{P}_{\Delta}}(d)\big)^{n-j} \nonumber\\
&=&\sum_{j=0}^{n}c_{j}\left(\bigoplus_{i=1}^{n+r} \mathcal{O}_{\mathbb{P}_{\Delta}}(D_i)\right)
\left(\sum_{i=1}^{n+r}c_1\big(\mathcal{O}_{\mathbb{P}_{\Delta}}(d_iD_i)\big)\right)^{n-j} \nonumber\\
&=& \sum_{j=0}^{n}c_{j}\left(\bigoplus_{i=1}^{n+r} \mathcal{O}_{\mathbb{P}_{\Delta}}(D_i)\right)
\left(\sum_{i=1}^{n+r} d_ih_i \right)^{n-j}.
\end{eqnarray*}
On the other hand, we have  
\begin{eqnarray*}
c\left(\bigoplus_{i=1}^{n+r} \mathcal{O}_{\mathbb{P}_{\Delta}}(D_i)\right)=
\prod_{i=1}^{n+r}c\big(\mathcal{O}_{\mathbb{P}_{\Delta}}(D_i)\big)=
\prod_{i=1}^{n+r}\Big(1+c_1\big(\mathcal{O}_{\mathbb{P}_{\Delta}}(D_i)\big)\Big)
=\prod_{i=1}^{n+r}(1+h_{i})=\sum_{j=0}^{n+r}\texttt{C}_{j}(h).
\end{eqnarray*}
Then 
\begin{eqnarray*}
c_{n}\big(\mathcal{T}\mathbb{P}_{\Delta}\otimes\mathcal{O}_{\mathbb{P}_{\Delta}}(d)\big)
=\sum_{j=0}^{n}\texttt{C}_j(h)\left(\sum_{i=1}^{n+r} d_ih_i \right)^{n-j}
=\sum_{j=0}^{n}\left\{\sum_{\left|k\right|=n-j}\binom{n-j}{k}\texttt{C}_j(h)
\prod_i \big(d_ih_i\big)^{k_i}\right\},
\end{eqnarray*}
and so we get  
\begin{eqnarray*}
\#\,\mathrm{Sing}(\mathcal{F})
=\sum_{j=0}^{n}\left\{\sum_{\left|k\right|=n-j}\binom{n-j}{k} \int_{\mathbb{P}_{\Delta}}^{\,orb} 
\texttt{C}_j(h)\cdot \prod_i \big(d_ih_i\big)^{k_i}\right\}.
\end{eqnarray*}
\end{proof}

\begin{cor} Let $\mathbb{P}_{\Delta}$ be an $n$-dimensional compact toric manifold with isolated singularities.
Let $\mathcal{F}$ be a one-dimensional generic foliation of degree $d=\sum_{i}d_{i} h_i$ on $\mathbb{P}_{\Delta}$. 
Then, $\mathcal{F}$ is singular if $\gcd\left\{d_i\right\} \nmid \texttt{C}_n(h)$.
\end{cor}
\begin{proof}
By Proposition \ref{pro1}, if $\mathrm{Sing}\left(\mathcal{F}\right)=\varnothing$, then we can write
$$0=\sum_{j=0}^{n}\left\{\sum_{\left|k\right|=n-j}\binom{n-j}{k}
\int_{\mathbb{P}_{\Delta}} \texttt{C}_j(h)\cdot \prod_i \big(d_i h_i\big)^{k_i}\right\} 
=\texttt{C}_n(h) + \sum_{i}a_i d_i,$$
for some $a_i\in\mathbb{Z}$. Then $\gcd\left\{d_i\right\} \mid \texttt{C}_n(h)$.
\end{proof}

Let us consider some examples in blow-ups:

\begin{exe} Let $\mathbb{P}_{\Delta}=\mathrm{Bl}_p(\mathbb{P}^n)$ be the blow-up of 
$\mathbb{P}^n$ at a point. We can put 
$$\mathrm{Pic}\big(\mathrm{Bl}_p(\mathbb{P}^n)\big)=\mathbb{Z}H\oplus\mathbb{Z}E,$$ 
where $H$ denotes the pullback of the hyperplane class in $\mathbb{P}^n$, and $E$ is the exceptional divisor. 
Since 
$$\mathcal{A}\big(\mathrm{Bl}_p(\mathbb{P}^n)\big)\simeq\frac{\mathbb{Z}\left[H,E\right]}
{\left\langle H\cdot E,H^n+(-1)^nE^n\right\rangle},$$
we have $H\cdot E=0$ and $H^n=(-1)^{n+1}E^n=1$, see \cite{EiHa}. 

The classes $h_i=\left[D_i\right]\in\mathcal{A}_{n-1}(\mathbb{P}_{\Delta})$, 
can be represented in $\mathrm{Pic}\big(\mathrm{Bl}_p(\mathbb{P}^n)\big)$ by the matrix
$$\bordermatrix{& h_1 & \cdots & h_n & h_{n+1} & h_{n+2} \cr
                &  1  & \cdots &  1  &    1    & 0  \cr
                &  -1  & \cdots &  -1  &    0    & 1 \cr}$$
\\								
\noindent That is, $h_1=\cdots=h_n=H-E$, $h_{n+1}=H$ and $h_{n+2}=E$, see \cite{Cox2, MaA}. 

For example, for $n=2$, let $\mathcal{F}$ be a holomorphic foliation of degree $d=(d_1,d_2)=d_1H+d_2E$ on $\mathrm{Bl}_p(\mathbb{P}^2)$.
Since $\texttt{C}_0(h)=1$, $\texttt{C}_1(h)=3H-E$ and $\texttt{C}_2(h)=3H^2-E^2=4$, we have
$$\#\,\mathrm{Sing}(\mathcal{F})=\sum_{j=0}^2\texttt{C}_j(h)d^{\,2-j}=d_1^2-d_2^2+3d_1+d_2+4.$$
\end{exe}

\begin{exe}
Let $\mathbb{P}_{\Delta}=\mathrm{Bl}_{p,q}(\mathbb{P}^3)$ be the blow-up of 
$\mathbb{P}^3$ at $2$ points. We can write 
$$\mathrm{Pic}\big(\mathrm{Bl}_{p,q}(\mathbb{P}^3)\big)=\mathbb{Z}H\oplus\mathbb{Z}E_1\oplus\mathbb{Z}E_2,$$ 
where $H$ denotes the pullback of the hyperplane class in $\mathbb{P}^3$, and $E_1,E_2$ the exceptional divisors. 
Here, we have the following intersection formulas: 
$$H^3=E_i^3=1,\,\,H\cdot E_i=0,\,\,\,\mbox{and}\,\,\,\,E_i\cdot E_j=0\,\,\,\mbox{for}\,\,\, i\neq j.$$

Now, let $\mathcal{F}$ be a one-dimensional foliation of degree $d=d_0H+d_1E_1+d_2E_2$ on $\mathrm{Bl}_{p,q}(\mathbb{P}^3)$.
Since 
$\texttt{C}_0(h)=1$, 
$\texttt{C}_1(h)=-K_{\mathrm{Bl}_{p,q}(\mathbb{P}^3)}=4H-2E_1-2E_2$, 
$\texttt{C}_2(h)=c_2(\mathbb{P}^3)=6H^2$; see \cite[p.609]{Grif}, 
and $\texttt{C}_3(h)=8$, 
we obtain 
$$\#\,\mathrm{Sing}(\mathcal{F})=\sum_{j=0}^3\texttt{C}_j(h)d^{\,3-j}=
d_0^3+d_1^3+d_2^3+4d_0^2-2d_1^2-2d_2^2+6d_0+8.$$
\end{exe}

\begin{exe}
Let $\pi:\mathbb{P}_{\Delta}=\mathrm{Bl}_L(\mathbb{P}^3)\rightarrow\mathbb{P}^3$ be the blow-up of 
$\mathbb{P}^3$ at a line $L\cong \mathbb{P}^1$. We can write 
$$\mathrm{Pic}\big(\mathrm{Bl}_L(\mathbb{P}^3)\big)=\mathbb{Z}H\oplus\mathbb{Z}E,$$ 
where $H$ denotes the pullback of the hyperplane class in $\mathbb{P}^3$, and 
$E\cong\mathbb{P}(\mathcal{N}_{\mathbb{P}^1/\mathbb{P}^3})\cong\mathbb{P}^1\times\mathbb{P}^1$ the exceptional divisor. 
Here, we have the following intersection formulas: 
$$H^3=1,\,\,H^2\cdot E=0,\,\,H\cdot E^2=-1,\,\,\,\mbox{and}\,\,\,\,E^3=-2.$$

Now, let $\mathcal{F}$ be a one-dimensional foliation of degree $d=d_1H+d_2E$ on $\mathrm{Bl}_L(\mathbb{P}^3)$.
Since 
$\texttt{C}_0(h)=1$, 
$\texttt{C}_1(h)=-K_{\mathrm{Bl}_L(\mathbb{P}^3)}=4H-E$, 
$\texttt{C}_2(h)=\pi^{\ast}(c_2(\mathbb{P}^3)+\eta_{_L})-\pi^{\ast}c_1(\mathbb{P}^3)\cdot E=7H^2-4H\cdot E$; see \cite[p.609]{Grif}, 
and $\texttt{C}_3(h)=6$, 
we obtain 
$$\#\,\mathrm{Sing}(\mathcal{F})=\sum_{j=0}^3\texttt{C}_j(h)d^{\,3-j}=
d_1^3-3d_1d_2^2-2d_2^3+4d_1^2-2d_2^2+2d_1d_2+7d_1+4d_2+6.$$
\end{exe}

\vspace{0,1cm}

In \cite[p.37]{Bru}, M. Brunella shows that if $\mathcal{F}$ is a regular holomorphic foliation on a Hirzebruch surface
(a rational normal scroll of dimension $2$), then $\mathcal{F}$ is a $\mathbb{P}^1$-fibration over $\mathbb{P}^1$.
For rational normal scroll of dimension greater $2$, we have: 

\begin{teo} \label{ttee}
There are no one-dimensional regular foliations in a rational normal scroll $\mathbb{P}_{\Delta}=\mathbb{F}(a)=
\mathbb{F}(a_1,\ldots,a_n)$, $a_1,\ldots,a_n\in\mathbb{Z}$, $n>2$. 
\end{teo}
\begin{proof}
Here $\mathrm{Pic}\big(\mathbb{F}(a)\big)=\mathbb{Z}L\oplus\mathbb{Z}M$, 
subject to the relations $L^2=0$, $M^n=\left|a\right|=\sum_{i=1}^n a_i$ and $M^{n-1}\cdot L=1$. Moreover
$h_1=h_2=L$ and $h_i'=h_{i+2}=-a_iL+M$ for all $i=1,\ldots,n$; see \cite{Re}. \\

Let $\mathcal{F}$ be a one-dimensional regular foliation of degree $d=\sum_{i}d_{i} h_i$ in $\mathbb{F}(a)$.
Then we can write
\begin{eqnarray*}
\#\,\mathrm{Sing}(\mathcal{F})&=&\sum_{j=0}^{n}\texttt{C}_j(h)(d_1h_1+d_2h_2)^{\,n-j} \\
&=&(-1)^n\sum_{j=0}^{n}(-1)^j\big((-d_1)h_1+(-d_2)h_2\big)^{\,n-j}\texttt{C}_j(h)=0.	
\end{eqnarray*}
Then, it follows from \cite[p.22]{Mig2} that 
\begin{eqnarray*}
0&=&-nd_1\left(-d_2-1\right)^{n-1}-2P(-d_2)+2(-1)^n -\left|a\right| d_2 \left(-d_2-1\right)^{n-1} \\
&=& (-1)^n\left(nd_1+\left|a\right|d_2\right)\left(d_2+1\right)^{n-1}-2P(-d_2)+2(-1)^n,
\end{eqnarray*}
where $P(t)=t\sum_{i=0}^{n-2}(-1)^{i}\binom{n}{i}t^{\,n-2-i}+(-1)^n(1-n)$.
In addition, the equation above implies two possibilities; $d_1=-2$ and $d_2=0$, or 
$d_1=-\frac{1}{n}(1+(-1)^{n+1}-2\left|a\right|)$ and $d_2=-2$. \\

Since $\mathcal{F}$ is induced by a polynomial vector field $X=\sum P_{i,j}\frac{\partial}{\partial z_{i,j}}$, where 
$\deg(P_{1,1})=\deg(P_{1,2})=(d_1+1,d_2)$ and $\deg(P_{2,j})=(d_1,d_2+1)$, with non-negative second coordinate, we
have $d_1=-2$ and $d_2=0$. Note that $\deg(P_{1,1})=\deg(P_{1,2})=(-1,0)$ implies $P_{1,1}=P_{1,2}=0$.
Therefore we can write 
$$X=\sum_{j=1}^n P_{2,j}\frac{\partial}{\partial z_{2,j}},$$
with $\deg(P_{2,j})=(-2,1)$. The polynomials $P_{2,j}$ are of the form $P_{2,j}=\sum_{k=1}^n Q_{j,k}z_{2,k}$, where $Q_{j,k}$ 
are homogeneous polynomials of degree $a_k-2$, in the variables $z_{1,1}$ and $z_{1,2}$.
In particular $X$ is tangent to the fibers ($\cong\mathbb{P}^{n-1}$) of $\mathbb{F}(a)$, which is a contradiction
because projective spaces do not admit one-dimensional regular foliations. 
\end{proof}

On quasi-smooth hypersurfaces invariant by one-dimensional foliations, we have the following proposition:

\begin{prop} \label{pro2} Let $\mathbb{P}_{\Delta}$ be an $n$-dimensional compact toric orbifold with isolated singularities.
Suppose $\mathcal{F}$ is a one-dimensional foliation of degree $d=\sum_{i}d_{i} h_i$
on $\mathbb{P}_{\Delta}$. Let $V \subset \mathbb{P}_{\Delta}$ be a quasi-smooth hypersurface invariant by $\mathcal{F}$ 
of degree $a=\sum_{i}a_{i} h_i$, such that $\mathcal{F}|_V$ only has isolated singularities. 
Then, the singular scheme of $\mathcal{F}|_V$ consists of
$$\#\,\mathrm{Sing}(\mathcal{F}|_V)
=\sum_{j=0}^{n-1} \sum_{k=0}^{j}(-1)^{k}\int_{\mathbb{P}_{\Delta}}^{\,orb}\texttt{C}_{j-k}(h)\,a^{k+1}\,d^{\,n-1-j},$$
points counted with multiplicity. In particular, the orbifold Euler characteristic of $V$ is 
$$\chi_{orb}(V)=\int_V^{\,orb} \texttt{C}_{n-1}(V)
=\sum_{k=0}^{n-1}(-1)^{k}\int_{\mathbb{P}_{\Delta}}^{\,orb}\texttt{C}_{n-1-k}(h)\,a^{k+1}.$$
\end{prop}

\begin{proof} It follows from theorem \cite[Theorem 1]{Mig} that
$$\#\,\mathrm{Sing}(\mathcal{F}|_V)=\int_{V}^{\,orb} c_{n-1}(\mathcal{T}V\otimes\mathcal{O}_{\mathbb{P}_{\Delta}}(d)|_{V}).$$
In order to calculate this integral, consider the following exact sequence
$$0 \longrightarrow \mathcal{T}V \longrightarrow \mathcal{T}\mathbb{P}_{\Delta}|_{V} \longrightarrow \mathcal{N}_{V} \longrightarrow 0\, ,$$
where $\mathcal{N}_V$ is the normal orbifold bundle. Then
$$c(\mathbb{P}_{\Delta})=c(V)c(\mathcal{N}_V)$$
and
\begin{eqnarray}\label{eq:1.10}
c_{j}(V)=c_{j}(\mathbb{P}_{\Delta})-c_{j-1}(V)c_{1}(\mathcal{N}_V),\,\,\, 1\leq j \leq n-1.
\end{eqnarray}

\vspace{0,1cm}

\noindent Moreover, by using the Euler formula we get 
$c_{j}(\mathbb{P}_{\Delta})=\texttt{C}_{j}(h)$, $1\leq j \leq n-1$.
On the other hand, since $\mathcal{N}_{V}=\mathcal{O}_{\mathbb{P}_{\Delta}}(V)|_{V}$, we have that
$c_{1}(\mathcal{N}_V)=a$. Then, replacing in (\ref{eq:1.10}), we obtain
$$c_{j}(V)=\sum_{k=0}^{j}(-1)^{k}\texttt{C}_{j-k}(h)\,a^k,\,\,\, 0 \leq j \leq n-1.$$
Therefore
\begin{align}
c_{n-1}(\mathcal{T}V\otimes\mathcal{O}_{\mathbb{P}_{\Delta}}(d)|_{V})
&=\sum_{j=0}^{n-1}c_{j}(V)c_{1}(\mathcal{O}_{\mathbb{P}_{\Delta}}(d))^{n-1-j}\nonumber\\
&=\sum_{j=0}^{n-1} \left\{\sum_{k=0}^{j}(-1)^{k}\texttt{C}_{j-k}(h)\,a^k\right\}\,d^{\,n-1-j},\nonumber
\end{align}
and so we get  
$$\#\,\mathrm{Sing}(\mathcal{F}|_V)
=\sum_{j=0}^{n-1} \sum_{k=0}^{j}(-1)^{k}\int_{V}^{\,orb}\texttt{C}_{j-k}(h)\,a^k\,d^{\,n-1-j}.$$
After the work of I. Satake (see \cite{BFK} and \cite{Sata}), Poincar\'e's duality holds, so that
$V$ can be seen as the Poincar\'e dual of $c_1 ([V])$. Then we obtain
$$\#\,\mathrm{Sing}(\mathcal{F}|_V)
=\sum_{j=0}^{n-1} \sum_{k=0}^{j}(-1)^{k}\int_{\mathbb{P}_{\Delta}}^{\,orb}\texttt{C}_{j-k}(h)\,a^{k+1}\,d^{\,n-1-j}.$$
\end{proof}

\vspace{0,1cm}

In \cite{MaDi} the authors provide a Poincar\'e-Hopf type theorem for 
non-compact complex manifold of the form $X\setminus D$, where $X$ is a complex compact
manifold and $D$ is a normal crossing divisor on $X$, see also \cite{Alu, Ii, NoKo, SiR}.
In particular, the authors provide a formula for the Euler characteristic $\chi(X\setminus D)$, defined by
$$\chi(X\setminus D)=\sum_{i=1}^n\dim_{\mathbb{C}} H^i_c(X\setminus D,\mathbb{C}).$$
The following result is a generalization of that obtained in \cite{MaDi}, specifically in the context of the complex projective space.

\begin{teo} \label{aa} Let $\mathbb{P}_{\Delta}$ be an $n$-dimensional compact toric orbifold with isolated singularities.
Suppose $\mathcal{F}$ is a one-dimensional foliation of degree $d=\sum_{i}d_{i} h_i$ on $\mathbb{P}_{\Delta}$
with isolated and non-degenerate singularities. 
Let $V \subset \mathbb{P}_{\Delta}$ be a quasi-smooth hypersurface invariant by $\mathcal{F}$ of degree $a=\sum_{i}a_{i} h_i$. Then,
the singular scheme of $\mathcal{F}|_{\mathbb{P}_{\Delta}\setminus V}$ consists of
$$\#\,\mathrm{Sing}(\mathcal{F}_{\mathbb{P}_{\Delta}\setminus V})
=\sum_{j=0}^{n}\sum_{i=0}^{n-j}(-1)^{i}\int_{\mathbb{P}_{\Delta}}^{\,orb}\texttt{C}_{n-j-i}(h)\,a^{i}\,d^{\,j},$$
points counted with multiplicity. In particular, if 
$V \subset \mathbb{P}_{\Delta}$ is a smooth hypersurface, we obtain 
$$\chi(\mathbb{P}_{\Delta}\setminus V)=\sum_{i=0}^n(-1)^{i}\int_{\mathbb{P}_{\Delta}}\texttt{C}_{n-i}(h)\,a^{i}.$$
\end{teo}

\begin{proof} By hypothesis, the singularities of $\mathcal{F}$ are non-degenerates. Then, the number 
$\#\left(\mathrm{Sing}(\mathcal{F})\cap V\right)$
corresponds to the sum of the numbers of Milnor $\#\,\mathrm{Sing}(\mathcal{F}|_V)$, and so 
$$\#\,\mathrm{Sing}(\mathcal{F}_{\mathbb{P}_{\Delta}\setminus V})=
\#\,\mathrm{Sing}(\mathcal{F}_{\mathbb{P}_{\Delta}})-
\#\,\mathrm{Sing}(\mathcal{F}|_{V}).$$

\noindent Then, by Propositions \ref{pro1} and \ref{pro2} we may write 

\begin{align*}
\#\,\mathrm{Sing}(\mathcal{F}_{\mathbb{P}_{\Delta}\setminus V})
&=\sum_{j=0}^{n}\int_{\mathbb{P}_{\Delta}}^{\,orb}\texttt{C}_j(h)d^{\,n-j}
-\sum_{j=0}^{n-1}\sum_{k=0}^{j}(-1)^{k}\int_{\mathbb{P}_{\Delta}}^{\,orb}\texttt{C}_{j-k}(h)\,a^{k+1}\,d^{\,n-1-j}\\
&=\sum_{j=0}^{n}\int_{\mathbb{P}_{\Delta}}^{\,orb}\texttt{C}_{n-j}(h)d^{\,j}
-\sum_{j=0}^{n-1}\sum_{k=0}^{n-1-j}(-1)^{k}\int_{\mathbb{P}_{\Delta}}^{\,orb}\texttt{C}_{n-1-j-k}(h)\,a^{k+1}\,d^{\,j}\\
&=\sum_{j=0}^{n}\int_{\mathbb{P}_{\Delta}}^{\,orb}\texttt{C}_{n-j}(h)d^{\,j}
+\sum_{j=0}^{n-1}\sum_{k=0}^{n-1-j}(-1)^{k+1}\int_{\mathbb{P}_{\Delta}}^{\,orb}\texttt{C}_{n-j-(k+1)}(h)\,a^{k+1}\,d^{\,j}\\
&=\sum_{j=0}^{n}(-1)^0\int_{\mathbb{P}_{\Delta}}^{\,orb}\texttt{C}_{n-j-0}(h)a^0d^{\,j}
+\sum_{j=0}^{n-1}\sum_{i=1}^{n-j}(-1)^{i}\int_{\mathbb{P}_{\Delta}}^{\,orb}\texttt{C}_{n-j-i}(h)\,a^{i}\,d^{\,j}\\
&=(-1)^0\int_{\mathbb{P}_{\Delta}}^{\,orb}\texttt{C}_{0}(h)a^0d^{\,n}
+\sum_{j=0}^{n-1}\sum_{i=0}^{n-j}(-1)^{i}\int_{\mathbb{P}_{\Delta}}^{\,orb}\texttt{C}_{n-j-i}(h)\,a^{i}\,d^{\,j}\\
&=\sum_{j=0}^{n}\sum_{i=0}^{n-j}(-1)^{i}\int_{\mathbb{P}_{\Delta}}^{\,orb}\texttt{C}_{n-j-i}(h)\,a^{i}\,d^{\,j}.
\end{align*}
\end{proof}

\begin{exe}
Let us consider an $2$-dimensional well formed weighted projective space 
$\mathbb{P}(\omega)=\mathbb{P}(\omega_0,\omega_1,\omega_2)$, with $\omega_0,\omega_1,\omega_2>1$. 
Choose $a_0, a_1, a_2\in \mathbb{C}^{\ast}$ such that  $a_{i}w_{j}\neq a_{j}w_{i}$ $\forall\,i\neq j$.
Here $h_i=w_ih$, where $h=c_{1}(\mathcal{O}_{\mathbb{P}(\omega)}(1))$.
Consider the one-dimensional foliation $\mathcal{F}$ of degree $d=0$ on $\mathbb{P}(\omega)$, induced in homogeneous coordinate by
$$X=\sum_{k=0}^{2}a_{k}z_{k}\frac{\partial}{\partial z_{k}}.$$

\noindent The local expression of $X$ over $U_i$ is
$$X|_{U_i}=\sum_{\stackrel{k=0}{k\neq i}}^2\Big(a_{k}-a_{i}\frac{w_{k}}{w_{i}}\Big)z_{k}\frac{\partial}{\partial z_{k}}.$$

\noindent Then $\mathrm{Sing}(X|_{U_i})=\{0\}$ and is non-degenerate, so

$$\mathrm{Sing}(\mathcal{F})=\left\{\left[1:0:0\right], \left[0:1:0\right],\left[0:0:1\right]\right\}=\mathrm{Sing}(\mathbb{P}(\omega)).$$

\noindent Note that $V=\left\{z_k=0\right\}\subset\mathbb{P}(\omega)$, $k\in\left\{0,1,2\right\}$, is a quasi-smooth hypersurface 
invariant by $\mathcal{F}$ of degree $a=w_k$. Therefore, we get
\begin{eqnarray*}
\#\,\mathrm{Sing}(\mathcal{F}_{\mathbb{P}(\omega)\setminus V})
&=&\frac{1}{\omega_k}=\frac{\texttt{C}_2(\omega)-\texttt{C}_1(\omega)\omega_k+\omega_k^2}{\omega_0\omega_1\omega_2} \\
&=&\sum_{i=0}^2(-1)^{i}\texttt{C}_{2-i}(\omega)\omega_k^{i}\int_{\mathbb{P}(\omega)}^{\,orb}h^2 \\
&=&\sum_{j=0}^2\sum_{i=0}^{2-j}(-1)^{i}\int_{\mathbb{P}(\omega)}^{\,orb}\texttt{C}_{2-j-i}(h)\,a^{i}\,d^{\,j}.
\end{eqnarray*}
\end{exe}

\section{Bott's residue formula on smooth weighted complete intersections}
 
Let $\mathbb{P}(\omega)$ be a well formed $n$-dimensional weighted projective space.
We say that $V\subset\mathbb{P}(\omega)$ is a smooth $N$-dimensional weighted complete intersection, 
when $V$ is the scheme-theoretic zero locus of $m=n-N$ weighted homogeneous 
polynomials $f_1,\ldots,f_m$ of degrees $a_1,\ldots a_m$.

\begin{teo} \label{teo1} 
Let $V=V(a_1,\ldots,a_m)\subset\mathbb{P}(\omega)$ be a smooth weighted complete intersection, and 
let $\mathcal{F}$ be a one-dimensional holomorphic foliation of degree $d$ on $\mathbb{P}(\omega)$.
Suppose that $\mathcal{F}$ leaves $V$ invariant and such that $\mathcal{F}|_V$ has only isolated singularities. 
Then 
$$\#\,\mathrm{Sing}(\mathcal{F}|_V)
=\frac{a_1\cdots a_m}{\omega_0\cdots\omega_n}\sum_{i=0}^{n-m}
\left\{\sum_{j=0}^i(-1)^j\texttt{C}_{\,i-j}(\omega)\cdot\mathcal{W}_j(a)\right\}d^{\,n-m-i}.$$
where $\texttt{C}_{k}(\omega)$ is the $k$th elementary symmetric function
of the variables $\omega_0,\ldots, \omega_{n}$, and $\mathcal{W}_k(a)$ is the Wronski (or complete symmetric function)
of degree $k$ in $m$ variables $a_1,\ldots,a_m$
$$\mathcal{W}_k(a)=\sum_{i_1+\cdots+i_m=k}a_1^{i_1}\ldots a_m^{i_m}.$$ 
\end{teo}

\begin{proof} It follows from theorem \cite[Theorem 1]{Mig} that
$$\#\,\mathrm{Sing}(\mathcal{F}|_V)=\int^{\,orb}_Vc_{n-m}\left(\mathcal{T}V(d)\right).$$

Similarly, as in the case of a smooth complete intersection in complex projective space,
the total Chern class of $V$ is
$$
c(V)=\frac{c\hspace{-0,5mm}\left(\mathcal{T}\mathbb{P}(\omega)|_V\right)}
{c\hspace{-0,5mm}\left(\mathcal{N}_{\mathbb{P}(\omega)/V}\right)}
=\frac{c\hspace{-0,5mm}\left(\bigoplus_{i=0}^n\mathcal{O}_{\mathbb{P}(\omega)}(\omega_i)|_V\right)}
{c\hspace{-0,5mm}\left(\bigoplus_{i=1}^m\mathcal{O}_{\mathbb{P}(\omega)}(a_i)|_V\right)},
$$

\vskip 0.2cm
\noindent see for instance \cite{WAWA}, see also \cite{Mig,EiHa,Mar}. Applying Whitney's formula, we get 

$$c(V)=\frac{\prod_{i=0}^n\left(1+\omega_i\,\eta\right)}{\prod_{i=0}^m\left(1+a_i\,\eta\right)},$$

\noindent where $\eta=c_1\hspace{-1mm}\left(i^{\ast}\mathcal{O}_{\mathbb{P}(\omega)}(1)\right)$, and $i:V\hookrightarrow\mathbb{P}(\omega)$ 
is the inclusion. Then 
\begin{eqnarray*}
c(V)&=&\left\{\sum_{j=0}^{\infty}(-1)^{j}\mathcal{W}_j(a)\,\eta^{\,j}\right\}
\left\{\sum_{k=0}^{n+1}\texttt{C}_{k}(\omega)\,\eta^k\right\} \\
&=& \sum_{i=0}^{n-m}\left\{\sum_{j+k=i}(-1)^{j}\mathcal{W}_j(a)\texttt{C}_{k}(\omega)\right\}\eta^{\,i} \\
&=& \sum_{i=0}^{n-m}\left\{\sum_{j=0}^{i}(-1)^{j}\mathcal{W}_j(a)\texttt{C}_{i-j}(\omega)\right\}\eta^{\,i},
\end{eqnarray*}
so that
$$c_i(V)=\left\{\sum_{j=0}^{i}(-1)^{j}\texttt{C}_{i-j}(\omega)\mathcal{W}_j(a)\right\}\eta^{\,i},\,\,\,\,\,\,0\leq i\leq n-m.$$
Hence,
\begin{eqnarray*}
c_{n-m}\hspace{-1mm}\left(\mathcal{T}V(d)\right)
&=&\sum_{i=0}^{n-m}c_i\hspace{-1mm}\left(\mathcal{T}V\right)
c_1\hspace{-1mm}\left(\mathcal{O}_{\mathbb{P}(\omega)}(d)|_V\right)^{n-m-i} \\
&=&\sum_{i=0}^{n-m}d^{\,n-m-i}c_i(V)\,\eta^{\,n-m-i} \\
&=&\sum_{i=0}^{n-m} \left\{\sum_{j=0}^i(-1)^j\texttt{C}_{\,i-j}(\omega)\mathcal{W}_j(a)\right\}d^{\,n-m-i}\,\eta^{\,n-m},
\end{eqnarray*}

\vskip 0.2cm
\noindent Finally, since $a_1\cdots a_m\,\eta^{\,m}$ is the Poincar\'e dual of $V$, we have

\begin{eqnarray*}
\#\,\mathrm{Sing}(\mathcal{F}|_V)&=&\int^{\,orb}_Vc_{n-m}\left(\mathcal{T}V(d)\right) \\
&=& a_1\cdots a_m\sum_{i=0}^{n-m}
\left\{\sum_{j=0}^i(-1)^j\texttt{C}_{\,i-j}(\omega)\mathcal{W}_j(a)\right\}d^{\,n-m-i}\int^{\,orb}_{\mathbb{P}(\omega)}
c_1\hspace{-1mm}\left(\mathcal{O}_{\mathbb{P}(\omega)}(1)\right)^n \\
&=& \frac{a_1\cdots a_m}{\omega_0\cdots\omega_n}\sum_{i=0}^{n-m}
\left\{\sum_{j=0}^i(-1)^j\texttt{C}_{\,i-j}(\omega)\mathcal{W}_j(a)\right\}d^{\,n-m-i}.
\end{eqnarray*}

\vskip 0.2cm
\noindent This finishes the proof of the theorem.
\end{proof}

Now, let us consider a Poincar\'e-type theorem for weighted complete intersections.
There exist works about Poincar\'e problem for complete intersection curves on projective spaces;
see for instance \cite{CCG,CJ,CE,Es,Mar}.

\begin{cor} \label{cccor}
Let $\mathcal{C}=V(a_1,\ldots,a_{n-1})\subset\mathbb{P}(\omega)$ be a smooth weighted complete intersection curve.
Suppose $\mathcal{F}$ is a one-dimensional holomorphic foliation of degree $d$ on $\mathbb{P}(\omega)$ 
which leaves $\mathcal{C}$ invariant. Then 
$$a_1+\cdots+a_{n-1}\leq d+\omega_0+\cdots+\omega_n.$$
\end{cor}

\begin{proof} By Theorem \ref{teo1}, we have 
\begin{eqnarray*}
\#\,\mathrm{Sing}(\mathcal{F}|_{\mathcal{C}})&=&\int^{\,orb}_{\mathcal{C}}c_1\left(\mathcal{T}\mathcal{C}(d)\right)\\
&=&\frac{a_1\cdots a_{n-1}}{\omega_0\cdots\omega_n}\sum_{i=0}^{1}
\left\{\sum_{j=0}^i(-1)^j\texttt{C}_{\,i-j}(\omega)\mathcal{W}_j(a)\right\}d^{\,1-i}\\
&=&\frac{a_1\cdots a_{n-1}}{\omega_0\cdots\omega_n}\big(d+\texttt{C}_1(\omega)-\mathcal{W}_1(a)\big).
\end{eqnarray*}
\\
\noindent Then, the theorem follows since $\#\,\mathrm{Sing}(\mathcal{F}|_{\mathcal{C}})\geq0$.
\end{proof}

More generally, we have 

\begin{teo} \label{teo3} 
Let $V=V(a_1,\ldots,a_m)\subset\mathbb{P}(\omega)$ be a smooth weighted complete intersection, 
and let $\mathcal{F}$ be a one-dimensional holomorphic foliation of degree $d$ on $V$. Then 
$$a_1+\cdots+a_m+\dim(V)-1\leq d+\omega_0+\cdots+\omega_n.$$
\end{teo}

\begin{proof} 
The pairing $\Omega^1_{V}\times\Omega^{n-m-1}_{V}\rightarrow\Omega^{n-m}_V$ induces
$\Omega^1_{V}\simeq\mathrm{Hom}_{\mathbb{C}}(\Omega^{n-m-1}_{V},\Omega^{n-m}_V)$.
On the other hand, $K_{V}=\mathcal{O}_{V}(\mathcal{W}_1(a)-\texttt{C}_1(\omega))$ and hence 
$$\mathcal{T}V(d)=(\Omega^1_{V})^{\vee}(d)=\Omega^{n-m-1}_{V}
\otimes(\Omega^{n-m}_{V})^{\vee}(d)=\Omega^{n-m-1}_{V}(d+\texttt{C}_1(\omega)-\mathcal{W}_1(a)).$$

By Bott-type vanishing formulas for $V$, $H^{0}(V, \Omega_{V}^{n-m-1}(t))=0$ for $t<n-m-1$; 
see \cite[Satz 8.11]{Fle}. Then we have $d+\texttt{C}_1(\omega)-\mathcal{W}_1(a)\geq n-m-1$.
\end{proof}

\bigskip

As was done for the Poincar\'e-Hopf index, we define the Baum-Bott index of a foliation $\mathcal{F}$
on an orbifold surface $S$ at an isolated singularity $p$ by 
$$\mathrm{BB}_p^{orb}\left(X\right)=\frac{1}{\#G_p}\mathrm{Res}_{\,\tilde{p}}
\left\{\frac{\big(\mathrm{Tr}(J\tilde{X})\big)^2}{\tilde{X}_1\tilde{X}_2}d\tilde{z}_1\wedge d\tilde{z}_2\right\},$$
where $\pi_p:(\widetilde{U},\tilde{p})\rightarrow(U,p)$ denotes a local quotient map at $p$:
$\tilde{X}= \pi_p^\ast X$,
$J{\tilde{X}}=\big(\frac{\partial\tilde{X}_i}{\partial\tilde{z}_j}\big)_{1\leq i,\,j\leq 2}$, and 
$\,\mathrm{Res}_{\,\tilde{p}}\displaystyle\left\{\frac{\mathrm{Tr}\big(J\tilde{X}\,\big)}
{\tilde{X}_1\tilde{X}_2}d\tilde{z}_1\wedge d\tilde{z}_2\right\}$ is Grothendieck's point residue.\\

The classification of regular foliations on projective surfaces of general type is open. 
Recall that a smooth weighted complete intersection surface $S=V(a_1,\ldots,a_{n-2})\subset\mathbb{P}(\omega)$ 
is of general type whenever $i_S=\sum a_i-\sum \omega_j>0$. In this context, we obtain the following result.

\begin{cor} \label{cwp}
Let $S=V(a_1,\ldots,a_{n-2})\subset\mathbb{P}(\omega)$ be a smooth weighted complete intersection surface, 
and let $\mathcal{F}$ be a one-dimensional foliation of degree $d$ on $\mathbb{P}(\omega)$ 
which leaves $S$ invariant. Suppose $\mathcal{F}|_S$ is generic. Then 
$$\sum_{p\in\mathrm{Sing}\left(\mathcal{F}|_S\right)}\mathrm{BB}_p^{orb}(\mathcal{F}|_S)
=\frac{a_1\cdots a_{n-2}}{\omega_0\cdots\omega_n}\,\,\big(d+\texttt{C}_1(\omega)-\mathcal{W}_1(a)\big)^2.$$
In particular, $\mathrm{Sing}(\mathcal{F}|_S)\neq\varnothing$. 
\end{cor}

\begin{proof} 
Since $\mathrm{tr}$ denotes the invariant symmetric polynomial $C_1$,
from theorem \cite[Theorem 1]{Mig}, and from the proof of Theorem \ref{teo1}, we have 
\begin{eqnarray*}
\sum_{p\in\mathrm{Sing}\left(\mathcal{F}|_S\right)}\mathrm{BB}_p^{orb}(\mathcal{F}|_S)
&=&\int_S^{\,orb} c_1^{\,2}\left(\mathcal{T}S\otimes\mathcal{O}_S(d)\right)\\
&=&\int_S^{\,orb} \big(c_1(\mathcal{T}S)+c_1(\mathcal{O}_S(d))\big)^2\\
&=&\int_S^{\,orb} \big((\texttt{C}_1(\omega)-\mathcal{W}_1(a))\eta+d\eta\big)^2\\
&=&\big(\texttt{C}_1(\omega)-\mathcal{W}_1(a)+d\,\big)^2\int_S^{\,orb}\eta^2\\
&=&\frac{a_1\cdots a_{n-2}}{\omega_0\cdots\omega_n}\,\big(d+\texttt{C}_1(\omega)-\mathcal{W}_1(a)\big)^2.
\end{eqnarray*}
\\
In particular, $\mathrm{Sing}(\mathcal{F}|_S)\neq\varnothing$. In fact, from Theorem \ref{teo3} we have 
$d+\texttt{C}_1(\omega)-\mathcal{W}_1(a)\geq 1$.
\end{proof}


\subsection{Induced distributions}

Set $n\geq m+2$.
Let $\mathbb{P}(\omega)$ be a well formed $n$-dimensional weighted projective space, 
and $V=V(a_1,\ldots,a_m)\subset\mathbb{P}(\omega)$ be a smooth weighted complete intersection.
By \cite[Theorem 3.2.4]{Do}, $\mbox{Pic}(V)\cong\mathbb{Z}$ if $n-m\geq3$. 
Furthermore $\omega_V\cong\mathcal{O}_V(\sum a_i-\sum \omega_j)$.

The inclusion map $i:V\hookrightarrow\mathbb{P}(\omega)$ induces a restriction map
$$i^{\ast}:H^{0}(\mathbb{P}(\omega),\Omega_{\mathbb{P}(\omega)}^1\otimes\mathcal{O}_{\mathbb{P}(\omega)}(d))
\rightarrow H^{0}(V,\Omega_{V}^1(d)).$$
By \cite[Proposition 5.4]{Ca}, if $n-m\geq3$, $2\leq a_1\leq a_2\leq\cdots\leq a_m$, $V(a_1,a_2,\ldots,a_i)$ is smooth 
for every $1\leq i\leq m$, and $d\leq a_1$, then $i^{\ast}$ is an isomorphism.

By Bott-type vanishing formulas for $V$, $H^{0}(V,\Omega_{V}^1(d))=0$ for $d<1$; 
see \cite[Satz 8.11]{Fle}. For the rest of this section, we will make the assumption that all distributions 
$\mathcal{D}$ on $V$ are induced by non-trivial sections $\omega$ in the image of $i^{\ast}$, 
and such that $d:=\mathrm{deg}(\mathcal{D})\geq1$.


\begin{obs} \label{ob1} 
Suppose $\mathcal{D}$ is a codimension one distribution of degree $d$ on $V$ with isolated singularities.
By Theorem \ref{teo1}, it is easy to see that
$$
\#\,\mathrm{Sing}(\mathcal{D})
=\frac{a_1\cdots a_m}{\omega_0\cdots\omega_n}\sum_{i=0}^{n-m}(-1)^i
\left\{\sum_{j=0}^i(-1)^j\texttt{C}_{\,i-j}(\omega)\cdot\mathcal{W}_j(a)\right\}d^{\,n-m-i}.
$$
\end{obs}

\begin{cor} \label{co1} 
Let $V=V(a_1,\ldots,a_m)\subset\mathbb{P}(\omega)$ be a smooth weighted complete intersection.
Suppose $\mathcal{D}$ is a codimension one generic distribution of degree $d$ on $V$. Set
$$\alpha(\omega,a)=\sum_{j=0}^{n-m}(-1)^j\texttt{C}_{\,n-m-j}(\omega)\cdot\mathcal{W}_j(a).$$
Then, $\mathcal{D}$ is singular if $d\nmid\alpha(\omega,a)$. Moreover,
$$\chi(V)=\dfrac{a_1\cdots a_m}{\omega_0\cdots\omega_n}\alpha(\omega,a).$$
\end{cor}

\begin{proof} By Remark \ref{ob1}, if $\mathrm{Sing}\left(\mathcal{D}\right)=\varnothing$, then we can write
$$0=\sum_{i=0}^{n-m}(-1)^i\left\{\sum_{j=0}^i(-1)^j\texttt{C}_{\,i-j}(\omega)\cdot\mathcal{W}_j(a)\right\}d^{\,n-m-i}
=(-1)^{n-m} \alpha(\omega,a) + \sum_{i=1}^{n-m}\lambda_i d^{\,i},$$
for some $\lambda_i\in\mathbb{Z}$. Then $d \mid \alpha(\omega,a)$.
The second part follows from $\chi(V)=(-1)^{n-m}\int^{\,orb}_Vc_{n-m}\left(\Omega_{V}^1\right)$; see \cite{Bla}.
\end{proof}

\vspace{0,1cm}

Now we present two examples that provide a numerical illustration of our formula above 

\begin{exe} $\mbox{}$
\vspace{0,1cm}
\begin{enumerate}
\item Consider the foliation $\mathcal{F}$ of degree $d=2k$ 
on $\mathbb{P}(\omega)=\mathbb{P}(1,1,1,k)$, $k>1$, defined by 
$$\Omega=kz_3z_0^{k-1}dz_0+kz_3z_1^{k-1}dz_1+kz_3z_2^{k-1}dz_2-\big(z_0^k+z_1^k+z_2^k\big)dz_3.$$ 
Note that $\mathrm{Sing}(\mathcal{F})=\left\{\bar{e}_3,\mathcal{C}\right\}$, where 
$\mathcal{C}=\left\{(z_0,z_1,z_2,0)\,|\,z_0^k+z_1^k+z_2^k=0\right\}$. Now consider the hypersurface $V=\left\{z_0=0\right\}$ of degree 
$a=1$. Then
$$\Omega|_V=kz_3z_1^{k-1}dz_1+kz_3z_2^{k-1}dz_2-\big(z_1^k+z_2^k\big)dz_3.$$
Set $\mu_k=\left\{\alpha\in\mathbb{C}^{\ast}:\alpha^k=1\right\}$. Then
$$\mathrm{Sing}\left(\mathcal{F}|_V\right)=\big\{\left[0:0:0:1\right], \left[0:1:i\alpha:0\right]\,:\,\,\alpha\in\mu_k\big\}.$$ 
In coordinates $\big(U_1,\left\{+1\right\}\big)\hspace{-1mm}:z_1\neq0$, we have 
$\Omega|_V=kz_3z_2^{k-1}dz_2-(1+z_2^k)dz_3$ and $\mathcal{I}^{orb}_{(0,i\alpha,0)}=1$ for all $\alpha\in\mu_k$. 
In coordinates $\big(U_3,\mu_k\big)\hspace{-1mm}:z_3\neq0$, we have $\Omega|_V=kz_1^{k-1}dz_1+kz_2^{k-1}dz_2$ and  
$\mathcal{I}^{orb}_{(0,0,0)}=\frac{(k-1)^2}{k}$. Since $\texttt{C}_1(\omega)=3+k$ and $\texttt{C}_2(\omega)=3+3k$, we obtain 
\begin{eqnarray*}
\#\,\mathrm{Sing}(\mathcal{F}|_V)&=&\frac{a}{\omega_0\,\omega_1\,\omega_2\,\omega_3}
\left\{d^{\,2}-(\texttt{C}_1(\omega)-a)d+\texttt{C}_2(\omega)-\texttt{C}_1(\omega)a+a^2\right\} \\ \\
&=& \frac{1}{k}\left\{(2k)^{\,2}-(2+k)(2k)+(3+3k)-(3+k)+1\right\} \\ \\
&=& \frac{1}{k}\left\{2k^2-2k+1\right\}=k+\frac{(k-1)^2}{k} \\ \\
&=& \sum_{p\in\mathrm{Sing}\left(\mathcal{F}|_V\right)}\mathcal{I}_p^{orb}\left(\mathcal{F}|_V\right).
\end{eqnarray*}
\vskip 0.1cm
\item Consider the distribution $\mathcal{D}$ of degree $d=\omega_0+\omega_1=\omega_2+\omega_3$ 
on $\mathbb{P}(\omega_0,\omega_1,\omega_2,\omega_3)$, defined by 
$$\Omega=\lambda\,\omega_0z_0dz_1-\lambda\,\omega_1z_1dz_0+\omega_2z_2dz_3-\omega_3z_3dz_2,\,\,\lambda\in\mathbb{C}^{\ast}.$$ 
Note that $\#\,\mathrm{Sing}(\mathcal{D})=\varnothing$.
Now consider the hypersurface $V=\left\{z_0=0\right\}$ of degree $a=\omega_0$. 
Then $\Omega|_V=\omega_2z_2dz_3-\omega_3z_3dz_2$ and $\mathrm{Sing}(\mathcal{D}|_V)=\left\{\bar{e}_1\right\}$.
Note that $\texttt{C}_1(\omega)=2d$ and $\texttt{C}_2(\omega)=d^{\,2}+\omega_0\omega_1+\omega_2\omega_3$.
Therefore, we get
\begin{eqnarray*}
\#\,\mathrm{Sing}(\mathcal{D}|_V)&=&\frac{a}{\omega_0\,\omega_1\,\omega_2\,\omega_3}
\left\{d^{\,2}-(\texttt{C}_1(\omega)-a)d+\texttt{C}_2(\omega)-\texttt{C}_1(\omega)a+a^2\right\} \\ \\
&=& \frac{1}{\omega_1\,\omega_2\,\omega_3}\left\{d^{\,2}-2d^{\,2}+\omega_0d+d^{\,2}+\omega_0\omega_1+\omega_2\omega_3
-2d\omega_0+\omega_0^{\,2}\right\} \\ \\
&=& \frac{1}{\omega_1\,\omega_2\,\omega_3}\left\{\omega_0\omega_1+\omega_2\omega_3-d\omega_0+\omega_0^{\,2}\right\}
=\frac{1}{\omega_1} \\ \\
&=& \mathcal{I}_{\bar{e}_1}^{orb}\left(\mathcal{D}|_V\right).
\end{eqnarray*}
\end{enumerate}
\end{exe}

\vspace{0,1cm}

\begin{cor} \label{coror}
Let $V\subset\mathbb{P}(1,1,1,k)$ be a smooth weighted hypersurface.
Then there are no codimension one regular distribution on $V$.
\end{cor}

\begin{proof} Suppose that there is a codimension one regular distribution $\mathcal{D}$ on $V$.
Set $a=\mathrm{deg}(V)$ and $d=\mathrm{deg}(\mathcal{D})$. Since $\texttt{C}_1(\omega)=3+k$, 
$\texttt{C}_2(\omega)=3+3k$ and $\#\,\mathrm{Sing}(\mathcal{D}|_V)=0$, we have
\begin{eqnarray*}
0&=&d^{\,2}-(\texttt{C}_1(\omega)-a)d+\texttt{C}_2(\omega)-\texttt{C}_1(\omega)a+a^2 \\
&=&d^{\,2}-(3+k-a)d+(3+3k)-(3+k)a+a^2,
\end{eqnarray*}
which implies 
$$k\left(a+d-3\right)=a^2+a(d-3)+(d^{\,2}-3d+3).$$ 

Note that this equation has no solutions for $a+d\leq3$.
Therefore we may assume that $a+d>3$, and so
$$k=\frac{a^2+a(d-3)+(d^{\,2}-3d+3)}{a+(d-3)}>a,$$ 
a contradiction; since $V$ is a smooth weighted hypersurface, $k|a$. 
\end{proof}

\vspace{0,1cm}

\begin{teo} \label{teo2} 
Let $V\subset\mathbb{P}(1,1,1,1,k)$ be a smooth hypersurface, and let
$\mathcal{D}$ be a regular distribution on $V$. Then $\mathrm{deg}(V)=k$ and $\mathrm{deg}(\mathcal{D})=2$.
\end{teo}
\noindent Note that $k=1$, $\mathrm{deg}(V)=1$ and $\mathrm{deg}(\mathcal{D})=2$ was expected.
\begin{proof} Suppose that there is a codimension one regular distribution $\mathcal{D}$ on $V$.
Set $a=\mathrm{deg}(V)$ and $d=\mathrm{deg}(\mathcal{D})$. Since $\texttt{C}_1(\omega)=4+k$, 
$\texttt{C}_2(\omega)=6+4k$, $\texttt{C}_3(\omega)=4+6k$ and $\#\,\mathrm{Sing}(\mathcal{D}|_V)=0$, we have
\begin{eqnarray*}
0&=&d^{\,3}-(\texttt{C}_1(\omega)-a)d^{\,2}+(\texttt{C}_2(\omega)-\texttt{C}_1(\omega)a+a^2)d
-(\texttt{C}_3(\omega)-\texttt{C}_2(\omega)a+\texttt{C}_1(\omega)a^2-a^3) \\
&=&d^{\,3}-(4+k-a)d^{\,2}+(6+4k-(4+k)a+a^2)d-(4+6k-(6+4k)a+(4+k)a^2-a^3),
\end{eqnarray*}
which implies 
$$k\left(a^2+(d-4)a+d^{\,2}-4d+6\right)=a^3+(d-4)a^2+(d^{\,2}-4d+6)a+(d^{\,3}-4d^{\,2}+6d-4).$$ 
This equation has no solutions for $d>2$ and $a>2$. In fact, since $d>2$, this implies $d^{\,3}-4d^{\,2}+6d-4>0$.  
On the other hand $d>2$ and $a>2$ implies $a^2+(d-4)a>a^2-2a>0$. Note that $d^{\,2}-4d+6>0$. Then
$$k=\frac{a^3+(d-4)a^2+(d^{\,2}-4d+6)a+(d^{\,3}-4d^{\,2}+6d-4)}{a^2+(d-4)a+(d^{\,2}-4d+6)}>a.$$
a contradiction; since $V$ is a smooth weighted hypersurface, $k|a$. Therefore we may assume that $d\leq 2$ or $a\leq 2$. \\

Let us consider the case $d\leq 2$. If $d=1$, then 
$$k=\frac{(a-1)^3}{a^2-3a+3}=a-\frac{1}{a(a-3)+3},$$ 
which implies $(a,k)=(1,0)$ (a contradiction) or $(a,k)=(2,1)$. Now, suppose that $(a,k)=(2,1)$. In \cite[Proposition 5.4]{Ca} 
it was proved that there is a natural restriction isomorphism 
$$H^{0}(\mathbb{P}(\omega), \Omega_{\mathbb{P}(\omega)}^1(t))\simeq H^{0}(V, \Omega_{V}^1(t)),$$ 
for $\mathrm{deg}(V)\geq2,t$. Then by Bott's formulas (see \cite{OSS}) we obtain  
$$\dim H^{0}(V, \Omega_{V}^1(1))=\dim H^{0}(\mathbb{P}(\omega), \Omega_{\mathbb{P}(\omega)}^1(1))=0,$$ 
and so $\mathcal{D}=0$, a contradiction. Finally, the case $d=2$ implies
$$k=\frac{a(a^2-2a+2)}{a^2-2a+2}=a.$$

It remains to consider the case $a\leq 2$. If $a=1$, then 
$$k=\frac{(d-1)^3}{d^2-3d+3}=d-\frac{1}{d(d-3)+3},$$ 
which implies $(d,k)=(1,0)$ (a contradiction) or $(d,k)=(2,1)$. Finally, the case $a=2$ implies $k=d$, and since $k|a$, we have 
$(d,k)=(1,1)$ (a contradiction as already seen) or $(d,k)=(2,2)$.
\end{proof}

In general, if $\mathbb{P}(1,\ldots,1,k)$ has even dimension, then there exist smooth weighted hypersurfaces $V$ of degree $k$
and regular codimension one distributions on $V$ of degree two; see next example. 

\begin{exe} \label{exee} Let $\mathbb{P}(\omega)=\mathbb{P}(1,\ldots,1,k)$ be an $n$-dimensional weighted projective space, $n=2m$.
Let $V\subset\mathbb{P}(\omega)$ be a smooth weighted hypersurface of degree $k$, defined by
$$F(z_0,\ldots,z_n)=z_n-\sum_{i=0}^{n-1}\lambda_iz_i^k,\,\,\lambda_i\in\mathbb{C}^{\ast}.$$
Now, consider the distribution of degree two $\mathcal{D}$ on $V$, defined by 
$$\Omega=\sum_{i=0}^{m-1}c_i\left(\omega_{2i}z_{2i}dz_{2i+1}-\omega_{2i+1}z_{2i+1}dz_{2i}\right),\,\,c_i\in\mathbb{C}^{\ast}.$$
Note that $\mathrm{Sing}\left(\mathcal{D}\right)=\left\{\bar{e}_n\right\}$ and 
$\mathrm{Sing}\left(\mathcal{D}|_V\right)=\varnothing$. Therefore $\mathcal{D}|_V$ is a regular distribution.  

Moreover, if we denote by $K_V$ the canonical divisor 
of $V\subset\mathbb{P}(w)$, then we conclude that
$$\mathcal{O}_V\hspace{-0.07cm}\left(-K_V\right)=\mathcal{O}_V\big(\Sigma w_i-k\big)
=\mathcal{O}_V\big(n)=\mathcal{O}_V(2)^{\otimes \frac{\mathrm{deg}(V)+1}{2}},$$
that is, the $1$-form $\Omega$ is a contact form. 
\end{exe} 

\begin{exe} \label{eexe} Consider $m=3$ in the above example. Now, consider the smooth weighted complete intersection threefold 
$V(1,1,k)=V\cap\left\{z_0=z_1=0\right\}$. It is easy to see that 
$\mathrm{Sing}\big(\mathcal{D}|_{V(1,1,k)}\big)=\varnothing$. \\

Moreover, since $\texttt{C}_1(\omega)=6+k$, $\texttt{C}_2(\omega)=15+6k$, $\texttt{C}_3(\omega)=20+15k$, $\mathcal{W}_1(a)=2+k$, 
$\mathcal{W}_2(a)=3+2k+k^2$ and $\mathcal{W}_3(a)=4+3k+2k^2+k^3$, by Remark \ref{ob1} we have
\begin{eqnarray*}
\#\,\mathrm{Sing}\big(\mathcal{D}|_{V(1,1,k)}\big)&=&8-4\left\{\texttt{C}_1(\omega)-\mathcal{W}_1(a)\right\}
+2\left\{\texttt{C}_2(\omega)-\texttt{C}_1(\omega)\mathcal{W}_1(a)+\mathcal{W}_2(a)\right\} \\
&&-\left\{\texttt{C}_3(\omega)-\texttt{C}_2(\omega)\mathcal{W}_1(a)+\texttt{C}_1(\omega)\mathcal{W}_2(a)-\mathcal{W}_3(a)\right\} \\
&=&8-16+12-4=0.
\end{eqnarray*}
\end{exe} 

\vspace{0,1cm}


\section{Bott's residue formula on complete intersections in smooth compact toric varieties.}

Let $\mathbb{P}_{\Delta}$ be an $n$-dimensional smooth compact toric variety.
Let $f_1,\ldots,f_m$ be quasi-homogeneous polynomials, with 
$a_i=\deg (f_i)=\sum_j a_{ij}h_j\in\mathcal{A}_{n-1}(\mathbb{P}_{\Delta})$,
and $V=V(a_1,\ldots,a_m)\subset\mathbb{P}_{\Delta}$ a smooth complete intersection.

\begin{teo} \label{tte} Let $V=V(a_1,\ldots,a_m)\subset\mathbb{P}_{\Delta}$ be a smooth complete intersection, and 
let $\mathcal{F}$ be a one-dimensional holomorphic foliation of degree $d$ on $\mathbb{P}_{\Delta}$, i.e., 
$\mathcal{F}$ is induced by a non-trivial section of $X\in H^{0}(V, \mathcal{T}\mathbb{P}_{\Delta}(d))$.
Suppose that $\mathcal{F}$ leaves $V$ invariant and $\mathcal{F}|_V$ has only isolated singularities. 
Then
$$\#\,\mathrm{Sing}(\mathcal{F}|_V)
=\sum_{i=0}^{n-m}\displaystyle\bigintsss_V\left\{\sum_{j+k=i}(-1)^{j}\mathcal{W}_j(a)\cdot\texttt{C}_{k}(h)\right\}d^{\,n-m-i},$$
where $\texttt{C}_{k}(h)$ is the $k$th elementary symmetric function
of the variables $h_1,\ldots, h_{n+r}$, and $\mathcal{W}_j(a)$ is the Wronski (or complete symmetric function)
of degree $j$ in $m$ variables $a_1,\ldots,a_m$,
$$\mathcal{W}_j(a)=\sum_{i_1+\cdots+i_m=j}a_1^{i_1}\ldots a_m^{i_m}.$$
In particular,
$$\chi(V)=\sum_{j+k=n-m}(-1)^j\int_V\mathcal{W}_j(a)\cdot\texttt{C}_k(h).$$
\end{teo}

\begin{proof}
From the exact sequence
$$0\rightarrow \mathcal{T}V \rightarrow \mathcal{T}\mathbb{P}_{\Delta}|_V\rightarrow \mathcal{N}_{\mathbb{P}_{\Delta}/V}\rightarrow 0,$$
with normal bundle $\mathcal{N}_{\mathbb{P}_{\Delta}/V}$, one has 
\begin{eqnarray} \label{ch}
c(V)=\frac{c\hspace{-0,5mm}\left(\mathcal{T}\mathbb{P}_{\Delta}|_V\right)}
{c\hspace{-0,5mm}\left(\mathcal{N}_{\mathbb{P}_{\Delta}/V}\right)},
\end{eqnarray}
where 
\begin{eqnarray} \label{ch1}
c\left(\mathcal{N}_{\mathbb{P}_{\Delta}/V}\right)=\prod_{i=1}^m
\big(1+c_1\left(\mathcal{O}_{\mathbb{P}_{\Delta}}(V_i)|_V\right)\big)=\prod_{i=1}^m(1+a_i),
\end{eqnarray} 
with $V_i=\left\{f_i=0\right\}$, see for instance \cite[p.59]{Fu}. On the other hand, using the generalized Euler's sequence, we have 
$$c\left(\mathcal{T}\mathbb{P}_{\Delta}\right)
=c\left(\mathcal{O}^{\oplus r}_{\mathbb{P}_{\Delta}}\right)c\left(\mathcal{T}\mathbb{P}_{\Delta}\right)
=c\left(\bigoplus_{i=1}^{n+r} \mathcal{O}_{\mathbb{P}_{\Delta}}(D_i)\right).$$
Note that $h_i=c_{1}(\mathcal{O}_{\mathbb{P}_{\Delta}}(D_i))=\left[D_i\right]$. Then
\begin{eqnarray} \label{ch2}
c\left(\mathcal{T}\mathbb{P}_{\Delta}\right)=
\prod_{i=1}^{n+r}c\big(\mathcal{O}_{\mathbb{P}_{\Delta}}(D_i)\big)=
\prod_{i=1}^{n+r}\Big(1+c_1\big(\mathcal{O}_{\mathbb{P}_{\Delta}}(D_i)\big)\Big)
=\prod_{i=1}^{n+r}(1+h_{i})=\sum_{k=0}^{n+r}\texttt{C}_{k}(h).
\end{eqnarray}

Substituting $(\ref{ch1})$, $(\ref{ch2})$ into $(\ref{ch})$ we obtain
\begin{eqnarray*}
c(V)&=&\left\{\sum_{j=0}^{\infty}(-1)^{j}\mathcal{W}_j(a)\right\}
\left\{\sum_{k=0}^{n+r}\texttt{C}_{k}(h)\right\} \\
&=& \sum_{i=0}^{n-m}\left\{\sum_{j+k=i}(-1)^{j}\mathcal{W}_j(a)\texttt{C}_{k}(h)\right\},
\end{eqnarray*}
so that
$$c_i(V)=\sum_{j+k=i}(-1)^{j}\mathcal{W}_j(a)\texttt{C}_{k}(h).$$
Hence,
\begin{eqnarray*}
c_{n-m}\hspace{-1mm}\left(\mathcal{T}V(d)\right)
&=&\sum_{i=0}^{n-m}c_i\hspace{-1mm}\left(V\right)
c_1\hspace{-1mm}\left(\mathcal{O}_{\mathbb{P}_{\Delta}}(d)|_V\right)^{n-m-i} \\
&=&\sum_{i=0}^{n-m}\left\{\sum_{j+k=i}(-1)^{j}\mathcal{W}_j(a)\texttt{C}_{k}(h)\right\}d^{\,n-m-i}.
\end{eqnarray*}

\vskip 0.2cm
\noindent Finally, we have

\begin{eqnarray*}
\#\,\mathrm{Sing}(\mathcal{F}|_V)&=&\int_Vc_{n-m}\hspace{-1mm}\left(\mathcal{T}V(d)\right) \\
&=& \sum_{i=0}^{n-m}\displaystyle\bigintsss_V\left\{\sum_{j+k=i}(-1)^{j}\mathcal{W}_j(a)\texttt{C}_{k}(h)\right\}d^{\,n-m-i}.
\end{eqnarray*}
\noindent This finishes the proof of the theorem.
\end{proof}

\begin{obs} Let $V=V(a_1,\ldots,a_m)\subset\mathbb{P}_{\Delta}$ be a smooth complete intersection, and 
let $\mathcal{D}$ be a codimension one distribution of degree $d$ on $V$ with isolated singularities, i.e., 
$\mathcal{D}$ is induced by a non-trivial section of $\omega\in H^{0}(V, \Omega^1_V(d))$. By Theorem \ref{tte}, 
it is easy to see that
$$\#\,\mathrm{Sing}(\mathcal{D})
=\sum_{i=0}^{n-m}(-1)^i\displaystyle\bigintsss_V\left\{\sum_{j+k=i}(-1)^{j}\mathcal{W}_j(a)\cdot\texttt{C}_{k}(h)\right\}d^{\,n-m-i}.$$
\end{obs}

Now, let us consider a Poincar\'e-type theorem for a smooth complete intersection curve on $\mathbb{P}_{\Delta}$.

\begin{defi}\cite{Fu}
If $V$ is a smooth subvariety of $\mathbb{P}_{\Delta}$, the $k$-degree of $V$, $\deg_k(V)$, is defined by 
$$\deg_k(V)=\int_Vc_{1}\left(\mathcal{O}_{\mathbb{P}_{\Delta}}(D_k)|_V\right)^{\dim(V)}.$$
\end{defi}
\vskip 0.1cm

\begin{cor} \label{ccrr}
Let $\mathcal{C}=V(a_1,\ldots,a_{n-1})\subset\mathbb{P}_{\Delta}$ be a smooth complete intersection curve.
Suppose $\mathcal{F}$ is a one-dimensional holomorphic foliation of degree $d$ on $\mathbb{P}_{\Delta}$ 
which leaves $\mathcal{C}$ invariant. Then 
$$\sum_{k=1}^{n+r} \left(\sum_{i=1}^{n-1}a_{ik}\right) \hspace{-0,5mm} \deg_k(\mathcal{C}) \leq \sum_{k=1}^{n+r} d_k \deg_k(\mathcal{C})
+\sum_{k=1}^{n+r} \deg_k(\mathcal{C}).$$
Alternatively,
$$\sum_{k=1}^{n-1} a_k \cdot\eta_{_\mathcal{C}} \leq d\cdot\eta_{_\mathcal{C}}-K_{\mathbb{P}_{\Delta}}\cdot\eta_{_\mathcal{C}},$$
\vskip 0.1cm
\noindent where $\eta_{_\mathcal{C}}$ is the Poincar\'e dual of $\mathcal{C}$.
\end{cor}

\vskip 0.1cm

\begin{proof}
By Theorem \ref{tte}, for $m=n-1$, and $i:\mathcal{C}\hookrightarrow\mathbb{P}_{\Delta}$ the inclusion, we have
\begin{eqnarray*}
\#\,\mathrm{Sing}(\mathcal{F}|_\mathcal{C})&=&
\sum_{i=0}^1\displaystyle\bigintsss_{\mathcal{C}}\left\{\sum_{j+k=i}(-1)^{j}\mathcal{W}_j(a)\texttt{C}_{k}(h)\right\}d^{\,1-i}\\ \\
&=&\int_{\mathcal{C}}i^{\ast}d+\int_{\mathcal{C}}i^{\ast}\texttt{C}_1(h)-\int_{\mathcal{C}}i^{\ast}\mathcal{W}_1(a) \\ \\
&=&\sum_k d_k \deg_k(\mathcal{C})+\sum_k \deg_k(\mathcal{C})
-\sum_i \sum_k a_{ik} \deg_k(\mathcal{C}).
\end{eqnarray*}
Therefore 
$$\sum_{k=1}^{n+r} \left(\sum_{i=1}^{n-1}a_{ik}\right) \hspace{-0,5mm} \deg_k(\mathcal{C}) \leq \sum_{k=1}^{n+r} d_k \deg_k(\mathcal{C})
+\sum_{k=1}^{n+r} \deg_k(\mathcal{C}).$$
\end{proof}

\begin{exe} For $\mathbb{P}_{\Delta}=\mathbb{P}^n$ this inequality can be improved. In fact, 
by positivity arguments, we have that $\#\,\mathrm{Sing}(\mathcal{F}|_\mathcal{C})>0$, when $d\geq1$; see \cite{Mar}. 
Then, since $\mathrm{Pic}(\mathbb{P}^n)=\mathbb{Z}\cdot\mathrm{H}$,
we see that $a_i=a_i\,\mathrm{H}$, $d=d\,\mathrm{H}$, and $\eta_{_\mathcal{C}}=a_1\cdots a_{n-1}\,\mathrm{H}^{n-1}$. Therefore,
$$\sum_{k=1}^{n-1}a_k\leq d+n.$$
This result was obtained in \cite{Mar} for projective spaces, where the author considered $d=\deg(\mathcal{F})-1$.
\end{exe}


\end{document}